 \documentclass[11pt]{article}
 \usepackage{bm}
 
 \newtheorem{theorem}{Theorem}[section]
 \newtheorem{lemma}[theorem]{Lemma}
 \newtheorem{corollary}[theorem]{Corollary}

 \newenvironment{proof}{{\it Proof}. }{\hfill\END\\[0.5ex]}
\newcommand{\igg}[1]{{{#1}}}

\newcommand{\ig}[1]{{{#1}}}

\usepackage{color} 
\usepackage[normalem]{ulem}
\usepackage{url}
\usepackage{graphicx}
\usepackage{amsmath}
\usepackage{amssymb}
\usepackage{algorithm}
\usepackage{algorithmic}
\usepackage{psfrag}
\usepackage{mathrsfs}

\newcommand{\R}{\mathbb{R}}

\newcommand{\be}{\begin{equation}}
\newcommand{\ee}{\end{equation}}
\newcommand{\cO}{\mathcal{O}}

 \makeatletter\@addtoreset{equation}{section}\makeatother

\newcommand{\rd}{\mathrm{d}}
\newcommand{\ri}{\mathrm{i}}

\newcommand{\wx}{\widehat{x}}
\newcommand{\wF}{\widehat{F}}

\newcommand{\END}{\hfill$\Box$}

\newenvironment{remark}%
{\refstepcounter{theorem}\noindent%
\textbf{Remark \thetheorem.}}%
{\hfill $\Box$ \\[0.75ex]}

\usepackage{fancyhdr}
\title{Filon-Clenshaw-Curtis rules
  for highly-oscillatory integrals with algebraic singularities and
  stationary points}

\author{V. Dom\'inguez\thanks{
Dep. Ingenier\'{\i}a Matem\'{a}tica e Inform\'{a}tica,
E.T.S.I.I.T.
Universidad P\'{u}blica de Navarra.
Campus de Tudela
31500 - Tudela (SPAIN), {email}:{\tt victor.dominguez@unavarra.es}} \and  I.G.
Graham\thanks{Department of Mathematical Sciences, University of Bath, Bath, BA2
7AY, United
Kingdom.
E-mail: {\tt I.G.Graham@bath.ac.uk}} \and T. Kim \thanks{Department of
Mathematical Sciences, University of Bath, Bath, BA2 7AY, United
Kingdom.
E-mail: {\tt T.Kim@bath.ac.uk}}}
\begin{document}

\maketitle

\begin{abstract}
In this paper we propose and analyse composite Filon-Clenshaw-Curtis quadrature
rules for integrals of the form $I_{k}^{[a,b]}(f,g) := \int_a^b f(x)
\exp(\mathrm{i}kg(x)) \rd x $,  where $k \geq  0$, 
$f$ may have integrable
singularities and $g$ may have  stationary points.   Our
composite rule is defined on a mesh with $M$ subintervals and requires $MN+1$
evaluations of $f$. It satisfies  an error estimate of the form $C_N
k^{-r} M^{-N-1 + r}$, where $r$ is determined by  the strength of any
singularity in $f$ and the order of any stationary points  in $g$ and  $C_N$ is
a constant which is independent of $k$ and $M$, but  depends on $N$. The
regularity  requirements on $f$ and $g$ are explicit in the error estimates. For
fixed $k$, the rate of convergence of the rule as $M \rightarrow \infty$   is
the same as would be obtained if $f$ was smooth.      Moreover, the quadrature
error decays   at least as fast as $k \rightarrow \infty$ as does   the original
integral $I_{k}^{[a,b]}(f,g)$.  For the case of nonlinear oscillators $g$,   
the algorithm requires the evaluation of   $g^{-1}$ at non-stationary points.
Numerical results demonstrate the sharpness of the theory. An  application  
to the implementation of boundary integral methods for the
high-frequency Helmholtz equation is given.  
\end{abstract}
 
\section{Introduction}
Oscillatory integrals of the form
\begin{equation}
\label{eq:theproblem}
 I_k^{[a,b]} (f,g) : = \int_{a}^b f(x)\exp({\rm i}kg(x))\ {\rm d}x  
\end{equation}
where  $f \in L^1[a,b]$ and {$k>0$} regularly appear in
applications. If $f$ and $g$ are smooth and $g'$ does not vanish  then
$ I_k^{[a,b]} (f,g) $  decays with at least $\mathcal{O}(k^{-1})$ 
as $k \rightarrow \infty$. The decay is faster than $\mathcal{O}(k^{-1})$ if $f$ and some 
of its derivatives vanish at both  end-points  $a,b$,  but    
is generally  slower 
if $f$ has a singularity or if   $g$ has a stationary point in $[a,b]$.  In practice one may be  interested in computing  \eqref{eq:theproblem}  efficiently and to controllable accuracy   
for a range of values of $k$ and for quite general  $f$ and $g$. 
 The {purpose} of this paper is to provide stable quadrature rules  for this task  
and to prove error estimates demonstrating the quality of the methods. We also  demonstrate the efficiency of our rules by applying them to  an example coming from boundary integral methods for  
high-frequency wave  scattering.    

Restricting first to the case  $g(x) = x$,  a recent paper   
\cite{DoGrSm:11}  studied the convergence of
Filon-Clenshaw-Curtis (FCC) rules for computing  the special case of 
\eqref{eq:theproblem}:
\begin{equation}
\label{eq:partcase}
 I_k (f) : = \int_{-1}^1 f(x)\exp({\rm i}kx)\ {\rm d}x\ .   
\end{equation}
These rules (denoted by $I_{k,N}(f)$),   approximate
\eqref{eq:partcase} (when $k\geq1/2$)
by replacing $f$ by its  
polynomial interpolant of degree $N$ at the Clenshaw-Curtis (or Chebyshev)  
points 
$t_{j,N} = \cos (jN/\pi)$, $j = 0, \ldots , N$. (When  $k<1/2$,  standard Clenshaw Curtis rules are used instead.)
A stability theory is given in \cite{DoGrSm:11} and a 
slight extension of the error estimates given in  \cite{DoGrSm:11}   (see Theorem
\ref{prop:convRule} below)   shows that, 
for  $r \in [0,2]$ and $m > \max\{ \frac12, \rho(r)\}$, 
\begin{equation}   \label{eq:1error} |I_k(f)  -I_{k,N} (f)  |\ \le \  C   
\left(\frac{1}{k}\right)^r  \left(\frac{1}{N}\right)^{m-\rho(r)}\|f_c\|_{H^{m}}\
,   \quad N \geq 1\ ,
\end{equation}
where $f_c(\theta) = f(\cos \theta)$, \ $ \rho(r) = r ,$    
$ r \in [0,1]$,  $ \rho(r) = 5r/2 -
3/2$,    $r
    \in [1,2]$ and $\Vert \cdot \Vert_{H^m}$ denotes the norm of the
    Sobolev space of order $m$ on $[-\pi, \pi]$. 
Thus  fast  convergence of the rule with respect to $N$ and decay of
the error  with order  
up to $\mathcal{O}(k^{-2})$ is obtained
if  $f$ is sufficiently regular.

However the  convergence rate  of the FCC rule is significantly impaired when
$f$ has one or more (integrable) singularities. Thus  in this paper we consider
composite  rules for \eqref{eq:theproblem} (first for $g(x) = x$),  obtained by
subdividing $[a,b]$ into a mesh with $M$ subintervals, chosen so that any 
singular points of $f$ coincide with mesh points.  We then construct  a
composite rule which uses the FCC rule on each   mesh subinterval not containing
 the singularities, and, on subintervals
containing the singularities, either  zero or  a very simple  two-point
rule is used, depending on the strength of the singularity. 
(See Section \ref{section:algorithm} for precise
description of the algorithm.)  To give a flavour of our  results,  
we show, for example,  that if  $f$ has  a singularity of form $\vert x - x_0
\vert^{\beta}$ for $x_0 \in [a,b]$ and $\beta \in (-1,1)\backslash \{0\}$, then
with suitable mesh refinement 
near $x_0$,   our rules have  error $E(f)$ which  satisfies the estimate
(see Theorem  \ref{thm:beta_all}):  
 \be \label{eq:sample_graded}
   E(f) \ \leq\   C_N
   \, \left(\frac{1}{k}\right)^r 
\, \left(\frac{1}{M}\right)^{N+1-r} \, \left\|f\right\|,
\ee
where $r \in [0, 1+ \beta]$ and the norm on $f$ is an appropriate weighted norm 
which takes  into account the singularity at $x_0$.  
The
estimate \eqref{eq:sample_graded} decays at least as fast with $k$  as does  the
corresponding integral \eqref{eq:theproblem}, since the latter  decays in
general  with  $\cO(k^{-r})$ where  $r =  \min\{1+\beta , 1\}$ -- see 
Lemmas \ref{lemma:betaNeg} -  \ref{remark:exact_integrals}. 



In order to prove   \eqref{eq:sample_graded}  (and its generalisations), in this
paper  a non-trivial extension of the estimate (\ref{eq:1error}) (quoted from
\cite{DoGrSm:11}) is first    obtained  in \S \ref{sec:CleCur}. Since the error
estimate \eqref{eq:1error} depends on the regularity of $f$ through the norm of
$f_c$ this estimate does not provide the correct scaling with respect to $h$
when it is transported to an interval of size $h$. Therefore in Theorem  
\ref{thm:filon-cc} we prove a variant of \eqref{eq:1error},  where $\Vert f_c
\Vert_{H^m}$ is replaced by the Chebyshev weighted norm of $f^{(m)}$. This new
estimate has the correct scaling behaviour,   as is shown in Theorem
\ref{th:FCC_ab}. In \S \ref{sec:Composite} we obtain the error analysis for the
composite FCC rule applied to \eqref{eq:theproblem} with  $g(x) = x$ when   $f$
has integrable singularities at a finite set of points, in particular  obtaining
error estimates of the form \eqref{eq:sample_graded}. In  \S
\ref{sec:Composite_Stationary} we further extend to the case where $g$ may have
a finite number of stationary points in $[a,b]$. The latter case can be reduced
to that studied in \S \ref{sec:Composite} provided we assume that the inverse of
$g$ is known (\ig{or is evaluated numerically}) 
on subintervals between  stationary points, and indeed the action
of $g^{-1}$ is required for the implementation of the algorithm. 
In \S
\ref{sec:Numerical} we give numerical experiments, \ig{utilising the
public domain code \cite{Do:code}}   
 which  indicate that as $M$ increases,  
the error decays with  $\cO(M^{-N-1})$,  provided the  parameters 
of the mesh are appropriately chosen, relative to the regularity of
$f$ and 
stationary points in $g$. Moreover, when $k$ increases the error decays roughly  
with  $\cO(k^{-r})$, where $r\in [0,1+\beta]$ indicating the sharpness of 
our theory. The numerical experiments also indicate  that applying 
the composite FCC on a graded mesh to integrals with singularities yields 
much more accurate results than applying FCC globally, and using the same  number of 
integrand evaluations.

In this paper we restrict our error estimates to  the case of  $k >0 $ for
convenience only; the rules also work well for all $k \in \mathbb{R}$ and the
error estimates can be easily  extended to that case (see, e.g., \cite[Corollary
2.3]{DoGrSm:11}). As is also shown in \cite{DoGrSm:11},  the FCC rules for
\eqref{eq:partcase} have a stable  implementation for  all $k$ and $N$ which,
via FFT, costs  $\cO(N\log N) $ operations. The composite rules presented here
require the evaluation of $f$ at $MN+1$ points. 

Although  oscillatory integration is  well-studied in the 
classical literature,  
some problems of interest to numerical 
analysts (even in 1D) still remain unsolved today. Thus this field has 
enjoyed   a recent upsurge of interest, partly because of its
importance in wave scattering applications.   (See \cite{ChGrLaSp:12}
and \cite{DoGrSm:11} for some  more detailed historical remarks.)
In particular,  the  construction and analysis of Filon-type  
methods has been examined  
in Iserles \cite{Is:04, Is:05}, Iserles and N\o rset \cite{IsNo:05}, 
Olver \cite{Ol:07}, Xiang \cite{xi:07} and Huybrechs and Olver \cite{HuOl:12}.
(Other  related methods include those of Levin-type  
\cite{Le:82,Ol:06,Ol:09} and those using numerical steepest descent 
\cite{HuVa:06}.)
In all these references, however,  the analysis  concentrates on 
accelerating the convergence as $k \rightarrow \infty$,  
generally assuming either that $f$ is sufficiently regular or $f$ has a 
particular type of singularity so that the moments 
(i.e. integrals \eqref{eq:theproblem} where $f$ is replaced by polynomials) 
can be computed using special functions. By contrast,  we propose 
Filon-type method for computing \eqref{eq:theproblem} 
where $f$ has algebraic singularities and $g$ may have 
stationary points and where the moments  can be obtained 
readily. Our method converges superalgebraically with respect to 
the number of quadrature points for any strength of singularity,  
provided the parameters of the mesh are chosen appropriately,  
and also converges with respect to $k$ at least as fast the 
integral itself converges to zero as $k \rightarrow \infty$.  
Our error estimates explicitly indicate the 
regularity requirements on $f$ and $g$. Other papers  
\cite{Me:09} and \cite{DoGrSm:11} provide analogous estimates for pure 
(non-composite) Filon rules, where $f$ and $g$ are  sufficiently regular
(and $g'\not=0$ ).  
But apart from these we know of no other 
contributions in this direction.  


Finally we mention  that our methods add something to traditional asymptotic
methods. 
The method of  stationary phase produces an accurate approximation  to an
integral if $k$ is sufficiently large, whereas our  methods work for all $k$ and
are superalgebraically convergent with respect to the number of function evaluations. Our methods  also yield 
a relative error which is superalgeraically convergent uniformly in $k$ and may
indeed even decay with $k$.

\section{ The Basic Filon Clenshaw-Curtis  Rule}
\label{sec:CleCur}

In this and the next section we will consider only the linear
oscillator $g(x) = x$ in \eqref{eq:theproblem}. (See \S \ref{sec:Composite_Stationary} for the case of nonlinear $g$.) Also, we introduce 
the notation
\begin{equation}
\label{eq:int_not}
I_{k}^{[a,b]}(f) \ := \  \int_{a}^b f(x)\exp({\rm i}kx)\ {\rm d}x \ .
\end{equation}
When $[a,b] = [-1,1]$ we will denote  
\eqref{eq:int_not} simply as $I_k(f)$.

\subsection{Integrals over the fundamental interval $[-1,1]$}

The FCC rule in its simplest form 
approximates   $I_k(f)$, by replacing  $f$ by 
its  algebraic polynomial  interpolant $Q_Nf$ at 
the Clenshaw-Curtis points $t_{j,N} := \cos(j \pi /N)\ , \  j = 0,
\ldots , N $ where $N \geq 1$. 
{Then for $k\ge 1/2$,  the rule is}  
 \begin{equation}
 \label{eq:therule}
 I_{k,N} (f) \ := \ \int_{-1}^1
 (Q_Nf)(x)\exp({\rm i}k x)\ {\rm d}x \ = \ \sum_{n=0}^{N}{}''  \
 \alpha_{n,N}(f)  \omega_n(k) \ , 
 \end{equation}
where, for $n \geq 0$,\quad 
$\displaystyle{\omega_n(k):=\int_{-1}^1 T_n(x)\exp({\rm i}k x)\ {\rm d}x\ , 
\ }$ 
 $T_n(x) = \cos(n \arccos(x))$ is the $n$th Chebyshev polynomial, and 
\begin{equation}
\label{eq:DCT} \alpha_{n,N}(f) \ =\ \frac2{N}{\sum_{j=0}^N}{}''
\ \cos\Big(\frac{j n \pi }N\Big) f\Big( t_{j,N}\Big)\ , \qquad n=0,\ldots, N\ .
\end{equation}
The notation $\sum{}''$ means that the first and last terms in
the sum  are multiplied by $1/2$. 


When $0<k<1/2$ the integrand in \eqref{eq:int_not} 
is non-oscillatory and we then apply the standard Clenshaw-Curtis rule:     
 \begin{equation}
 \label{eq:therule:02}
 I_{k,N} (f) \ := \ \int_{-1}^1
 (Q_N f_k)(x)\ {\rm d}x \ = \ \sum_{n=0}^{N}{}''  \
 \alpha_{n,N}(f_k)  \omega_n(0) \ ,\qquad f_k(x):=f(x)\exp({\rm i}k x)\ . 
 \end{equation}
We point out that $\omega_n(0)=2/(1-n^2)$ if $n$ is even, $0$ otherwise. 
For {$k\geq 1/2 $}, the computation of $\omega_n(k)$ turns out to be more delicate.
However in \cite{DoGrSm:11} a stable and  efficient scheme for computing {these}
weights  is presented. After an initial application of the discrete
cosine transform (via FFT, costing $\mathcal{O}(N \log N)$ operations),  
the rule \eqref{eq:therule}-\eqref{eq:therule:02} can then be applied to any $f$ in 
an additional $\mathcal{O}(N)$ operations - see \cite{DoGrSm:11} for
more detail.    
 In the error analysis in this section, we shall make use of the Sobolev space
$H^m$ of $2 \pi-$periodic functions with the norm 
\begin{equation}
 \label{eq:sobolevnorm}
 \|\varphi\|_{H^m}^2:=|\widehat{\varphi}(0)|^2+\sum_{\mu \ne 0}
|\mu|^{2m}|\widehat{\varphi}(\mu)|^2,\qquad
\widehat{\varphi}(\mu):=\frac{1}{2\pi}\int_{-\pi}^\pi \varphi(\theta)\exp(-{\rm
i}\mu\theta)\ {\rm d}\theta. 
\end{equation}
For   $\phi \in H^0 =  L^2[-\pi,\pi]$   
and  any $J \geq 0$, we introduce the truncated Fourier series:
$$
\left( \mathcal{S}_J \phi \right)(\theta) = \sum_{\mu=-J}^{ J}
\hat{\phi}(\mu) \exp(\ri \mu \theta)   $$
which converges  to $\phi \in H^0$  \text{ as}  
$J \rightarrow \infty$.  
In fact  when   $\phi$ is even,  we have 
\be
\left( \mathcal{S}_J \phi\right)(\theta) := \hat{\phi}(0)+ 2\sum_{\mu=1}^{
J} \hat{\phi}(\mu) \cos(\mu \theta) \ , \quad \text{where} \quad \
\hat{\phi}(\mu) = \frac{1}{\pi} \int_{0}^{\pi}
\phi(\theta) \cos(\mu \theta) d\theta \ , \label{eq:soper}
\ee
where the second  sum is void if $J=0$. 

Introducing the notation $f_c(\theta) = f(\cos \theta)$, and 
$ \rho(r) = r ,$    \ 
$ r \in [0,1]$, and $ \rho(r) = 5r/2 -
3/2$,   \  $r
    \in [1,2]$,  
the  following theorem is then  a minor extension of  
  \cite[Theorem 2.2]{DoGrSm:11}.
\begin{theorem}
\label{prop:convRule}
There exists a constant $C>0 $ such that, for all  $r\in[0,2]$ and  
all integers  $m  > \max\{1/2,\rho(r)\}$, the estimate 
\begin{equation}
\label{eq:cgcethm}
  |I_k(f)  -I_{k,N} (f)  |\ \le \  C
\left(\frac{1}{k}\right)^r\left(\frac{1}{N}\right)^{m-\rho(r)}\|f_c\|_{H^{m}}\
, \quad k \geq 1,  \quad N \geq 1
\end{equation}
holds when $f_c  \in H^m$. 

\end{theorem}

\begin{proof}
In  \cite{DoGrSm:11} the estimate \eqref{eq:cgcethm} 
was obtained for {$k\ge 1/2$},  for
$r = 0,1,2$. Hence for any $\theta  \in [0,1]$ we can write 
\begin{equation}\label{eq:interp} \vert I_k(f) - I_{k,N}(f)\vert \  = \ \vert I_k(f) -
I_{k,N}(f)\vert^{\theta}  \vert I_k(f) -
I_{k,N}(f)\vert^{1-\theta}\ . \end{equation} Then using estimate \eqref{eq:cgcethm}
with $r = 1$ (respectively $ r = 2$)  to bound the first (respectively
second) factor  on the right hand-side of \eqref{eq:interp} we obtain   
$$ \vert I_k(f) - I_{k,N}(f)\vert \ \leq \ C \
\left(\frac{1}{k}\right)^{2- \theta} 
\left(\frac{1}{N}\right)^{m- 7/2 + 5 \theta/2}\|f_c\|_{H^{m}}\ .
$$
Then setting  $\theta = 2 - r$ we obtain \eqref{eq:cgcethm} for $r
\in [1,2]$ and for $k \geq 1$. An even simpler interpolation  argument obtains the estimate for
$r \in [0,1]$. For $k<1/2$, i.e., for the classical Clenshaw-Curtis rule,  
the case $r=0$ follows by the same arguments used in  \cite[Theorem 2.2]{DoGrSm:11}  to prove \eqref{eq:cgcethm}. For $r\in (0,2)$ the result is obvious since $k$ is bounded.
 \end{proof} 

Theorem \ref{prop:convRule} ensures arbitrarily high convergence  
for the FCC rule  as   $N \rightarrow \infty$, 
provided $f$ is sufficiently
smooth. When  $f$ is not
smooth it is better to  apply the FCC rule in a composite
fashion on meshes graded suitably towards the singular point(s). 
These composite
rules then typically have fixed $N$ and converge as  the 
subinterval size  shrinks to zero. 
  In order to obtain good error estimates for the  composite rules we need to modify
the error estimate  in Theorem \ref{prop:convRule} so that  derivatives of $f$
rather than derivatives of $f_c$ appear in the bound. This will be done in
Theorem \ref{thm:filon-cc}, which in turn is used to obtain Theorem
\ref{th:FCC_ab},  showing  how the error of the FCC rule, when applied on
an arbitrary interval,  depends on the length of the interval.   In
order to prove Theorem \ref{thm:filon-cc} we first need two  lemmas. 

\begin{lemma} \label{lemma:ftprop} Let $f$
be such that $(f')_c \in L^1[-\pi, \pi]$. 
 Then  
\be\widehat{f_c}(\mu) =
\frac{1}{2\mu}
\left[\widehat{(f')_c}(\mu-1)-\widehat{(f')_c}(\mu+1)\right] \ , \quad
\text{for} \quad  \mu \neq 0. 
\label{eq:fourierprop}\ee
\end{lemma}
\begin{proof}
Since $f_c$ is even we use \eqref{eq:soper} and 
integrate by parts to obtain
\begin{eqnarray*}\widehat{f_c}(\mu) &=& \frac{1}{\pi} \int_{0}^{\pi}
f_{c}(\theta) \cos(\mu\theta) \rd\theta\  = \  
-\frac{1}{\pi\mu} \int_{0}^{\pi} \left(f_c\right)'
(\theta)\sin(\mu\theta) \rd\theta\\
&=&
\frac{1}{\mu\pi} \int_{0}^{\pi} f'(\cos\theta) \sin\theta
\sin(\mu\theta) \rd\theta\\&=& \frac{1}{2\mu \pi} \int_{0}^{\pi}
\left(f'\right)_c(\theta)\left[\cos( (\mu-1)\theta) - \cos((\mu+1)\theta)\right]
\rd\theta \ ,
\end{eqnarray*} and  the result follows. 
\end{proof}

Using Lemma \ref{lemma:ftprop}, we now
estimate  
the error in the truncated Fourier cosine series of $f_c$.   
\begin{lemma} \label{lemma:ftnorm}
For    all $0\leq m \leq N+1$    
there exist  constants $\sigma_{m,N}>0$ such that 
\be \label{eq:lemma_statement}
\|(I-\mathcal{S}_N)f_c\|_{H^{m}} \ \leq\  
\sigma_{m,N}   \|(f^{(m)})_c\|_{H^{0}}\ . 
\ee
\end{lemma}
\begin{proof} Since $\mathcal{S}_N $ is the orthogonal projection of   $H^0$
onto $\mathrm{span}\{ \exp(\ri j\theta): 0 \leq \vert j \vert \leq N\}$, the result
is trivial for $m = 0$. So let us assume now that  $m \geq 1$.  
Since $f_c$ is even,   
from (\ref{eq:soper}) we have,  \text{for all} $ J \geq 0$, 
\[
(I-\mathcal{S}_J)f_c = 2
\sum_{\mu \geq J+1} \widehat{f_c}(\mu) \cos(\mu \theta),\quad\text{and} \quad  {
\|(I-\mathcal{S}_J)f_c\|^2_{H^{m}} = 2 \sum_{\mu\geq  J+1 } \mu^{2m}
\big|\widehat{f_c}(\mu)\big|^2 . } 
\] 
Then, using Lemma \ref{lemma:ftprop}, we obtain 
\begin{eqnarray}
\left\|\left(I-\mathcal{S}_J\right)f_c\right\|^2_{H^{m}}
&\leq& \frac{1}{2} \sum_{\mu \geq J+1} \mu^{2 m-2} \left| \widehat{({f'})_c}(\mu-1)-
\widehat{(f')_c}(\mu+1)\right|^2 \nonumber\\ &\leq&   \sum_{\mu \geq J+1}
\mu^{2 m-2} \left| \widehat{(f')_c}(\mu-1) \right|^2+
 \sum_{\mu \geq J+1} \mu^{2 m-2} \left|
  \widehat{(f')_c}(\mu+1)\right|^2\nonumber\\
& = & \sum_{\mu \geq J}
(\mu+1)^{2 m-2} \left| \widehat{(f')_c}(\mu) \right|^2+  \sum_{\mu \geq J+2}
(\mu-1)^{2 m-2} \left| \widehat{(f')_c}(\mu)\right|^2\ \nonumber 
\\
& \leq & 2 \sum_{\mu \geq J}
\left({\mu+1}\right)^{2 m-2}  \left| \widehat{(f')_c}(\mu) \right|^2\ .
\label{eq:JtoJm1}
\end{eqnarray}
Hence,   in the case  $J\geq 1$, we have, 
\begin{eqnarray}
\left\|\left(I-\mathcal{S}_J\right)f_c\right\|^2_{H^m} &\leq&  
2 \left( \frac{J+1}{J} \right)^{2m-2}   \sum_{\mu \geq J} \mu^{2 m-2}  \left| \widehat{(f')_c}(\mu)\right|^2 \ \nonumber\\
& = &  \left( \frac{J+1}{J} \right)^{2m-2} \Vert (I -
\mathcal{S}_{J-1} ) (f')_c\Vert_{H^{m-1}}^2\ . 
\label{eq:2.star}
\end{eqnarray}
Using this identity $m-1$ times  (recalling that $m\leq N+1$), 
 we obtain
\begin{eqnarray*}
& &  \mbox{\hspace{-0.5in}}  \Vert (I-\mathcal{S}_N)f_c\Vert_{H^m} \ \leq \  \\
& \quad & \left(\frac{N+1}{N} \right)^{m-1}\left(\frac{N}{N-1}\right)^{m-2} \ldots \left(\frac{N-m+3}{N-m+2}\right)
\Vert (I - \mathcal{S}_{N-m+1})(f^{(m-1)})_c\Vert_{H^1}\ ,
\end{eqnarray*}
which we write as 
\begin{equation}
\label{eq:onestep}
\Vert (I-\mathcal{S}_N)f_c\Vert_{H^m} \ \leq \ \sigma_{m,N}
\Vert (I - \mathcal{S}_{N-m+1})(f^{(m-1)})_c\Vert_{H^1}\  . 
\end{equation}
Now, if $m <N+1$, we can use {\eqref{eq:2.star}} one more time and use
the fact that $\mathcal{S}_{N-m}$ is an orthogonal projection on $H^0$
to obtain the required result. On the other hand if  $m = N+1$ we can
use the fact that 
$$
\Vert (I - \mathcal{S}_0 )(f^{(m-1)})_c\Vert_{H^1} \ \leq \
   \Vert (f^{(m)})_c\Vert_{H^0}\ ,
$$
which is easily obtained from 
\eqref{eq:JtoJm1}, to deduce (\ref{eq:lemma_statement}). 
\end{proof}

\begin{remark} {To estimate the constants $\sigma_{m,N}$, note that in  
each pair of terms in the product, we can cancel the denominator 
in the left-hand term with the numerator in the right-hand term to obtain, for $m \geq 1$,  
$$\sigma_{m,N}  \ = \ \frac{(N+1)^{m-1}}{N(N-1) \ldots (N-m+2)}\ = \ \prod_{j=1}^{m-1} \left(\frac{N+1}{N+1 - j}\right) , $$
(with the product being interpreted as $1$ when $m = 1$). 
Thus  for fixed $m \geq 1$,  $\sigma_{m,N}\rightarrow 1$ as $N \rightarrow \infty$. 
Moreover letting $m$ grown with $N$ (for example $m = N+1$), we have 
$$\sigma_{N+1,N} = \frac{(N+1)^N}{N!} \ =\ \frac{N^N}{N!}\left(1 + \frac{1}{N}\right)^N \ \sim  \frac{e^{N+1}}{\sqrt{2\pi N}}, $$ 
where the last relation is obtained using Stirling's formula.  
Since in this paper we will use  fixed order methods (i.e. $N$ fixed) and 
obtain convergence for composite methods as the mesh size shrinks, the growth 
of $\sigma_{N+1,N}$ is not of essential importance to us here. However if we wanted to use $hp$ quadrature then this growth would be important 
and would need to be cancelled by suitable 
decay of the derivatives of $f$ in order to obtain convergence.   Estimates of this type are in \cite{Me:09}. 
}
\end{remark}

Now in Theorem  \ref{thm:filon-cc} below we will obtain the analogue of
Theorem \ref{prop:convRule}, but with (appropriate  weighted norms of)
derivatives of $f$, rather than
$f_c$ on the right-hand side.   
For any integer $m \geq 0$, and a function $f$ defined on $[a,b]$,  we introduce
the weighted  seminorm
\be \label{eq:weight} 
\vert f \vert_{H^m_w[a,b] } \ : = \ \left\{ \int_{a}^b \frac{\vert f^{(m)}(x)
  \vert^2}{\sqrt{(b-x) (x-a)}} \mathrm{d}x \right\}^{1/2}\ .  \ee 
When  $[a,b] = [-1,1]$ we just write $\vert \cdot  \vert_{H^m_w }$ and
we note that 
\be 
\vert f \vert_{H^m_w }
\  = \ \left\{ \int_{0}^\pi \vert
(f^{(m)})_c (\theta)\vert^2 \,  \mathrm{d}\theta \
\right\}^{1/2}={{\sqrt{\pi}}\big\|\big(f^{(m)}\big)_c\big\|_{H^0}} .    \label{eq:weightedc} 
\ee

\begin{theorem}
\label{thm:filon-cc}
Let  
$r \in [0,2]$ and $0 \leq m \leq N+1$. There exist constants $\sigma_{m,N}'$
such that 
\be
   \left|I_{k}(f) - I_{k,N}(f) \right| \ \leq
   \  \sigma_{m,N}'   \left(\frac{1}{k}\right)^r
   \left(\frac{1}{N}\right)^{m-\rho(r)} \vert f\vert_{H^m_w} 
   \ , 
\ee
   when $\vert f\vert_{H^m_w} < \infty $. {Moreover $\sigma_{m,N}' = C 
\sigma_{m,N}$ with $C$ independent of $m,N$.}  
\end{theorem}

\begin{proof} 
Note that if  $\vert f\vert_{H^m_w} < \infty $ then $f \in L^2[-1,1]$
and so we can  define an  
 algebraic polynomial $p$ of degree
$N$ by 
\be
p(x) = \widehat{f_c} (0) +2\sum_{n=1}^{ N}
\widehat{f_c} (n) T_{n} (x) \ .\label{eq:toper}
\ee
Clearly (recalling \eqref{eq:soper}),  
$
p_c(\theta) = \left( \mathcal{S}_N
f_c\right)(\theta)$, \text{for all} $\theta \in [-\pi,\pi] \ .
$
Since   $I_{k,N}$ is exact for all
polynomials of degree up to $N$, we have, using Theorem
\ref{prop:convRule},   
\begin{eqnarray*}|I_{k}(f) -
I_{k,N}(f)|&=& \left| I_{k} \left(f-p\right) -
I_{k,N}\left(f- p \right) \right|\\&\leq&
C \left(\frac{1}{k}\right)^r \left(\frac{1}{N}\right)^{m-\rho(r)}
\left\|(f- p )_c\right\|_{H^{m}} \\ & = &
C\left(\frac{1}{k}\right)^r \left(\frac{1}{N}\right)^{m-\rho(r)}\left
\|(I - \mathcal{S}_{N})f_c\right\|_{H^{m}} .\end{eqnarray*}
Then, using Lemma \ref{lemma:ftnorm}, 
\begin{eqnarray}\left|I_{k}(f) - I_{k,N}(f)\right| &\leq&
C \sigma_{m,N} \left(\frac{1}{k}\right)^r\left(\frac{1}{N}\right)^{m-\rho(r)} \,  
\|(f^{(m)})_c\|_{H^0}  
\end{eqnarray}
and the result follows from \eqref{eq:weightedc}.
\end{proof}

\subsection{Integrals over $[a,b]$}
\label{subsec:ab} 
Now we consider the  integral \eqref{eq:int_not} for general
$[a,b]$. 
To apply the FCC
quadrature, we first  transform the integral using the following
linear change of variables:
\be
x = c + ht, \ \ t  \in [-1,1], \ \ \ \ \ \ \ \ \ \text{where} \ \ \
c := \frac{b+a}{2}, \ \ \text{and} \ \ h := \frac{b-a}{2} \ .\label{eq:abtounit}
\ee
Then we may write 
\be
I_{k}^{[a,b]}(f) 
= h \exp\left({{\rm i}}kc \right) I_{\tilde{k}}
(\tilde{f}),\label{eq:trans} 
\ee
where $ \tilde{k} = h k $ and $\tilde{f}$ is the  function on
$[-1,1]$: 
\be  
 \tilde{f}(t) = f\left(c + ht\right) \ , \quad t \in [-1,1] \
 .\label{eq:deftilde} 
\ee
Then we apply the  quadrature rule  \eqref{eq:therule}-\eqref{eq:therule:02} to the
integral on the right-hand side of \eqref{eq:trans} to obtain the approximation  

\be  \ I_{k,N}^{[a,b]}(f) := h \exp \left({{\rm i}} k c \right)
I_{\tilde{k},N}(\tilde{f}) \ \approx \ I_k ^{[a,b]}(f) \
.  \label{eq:transquad} \ee 



The following theorem  is the corresponding extension of Theorem
\ref{thm:filon-cc}. 

\begin{theorem}
\label{th:FCC_ab}
 Let $r \in [0,2]$ and  $0\leq m \leq N+1$.  Then, when $\vert f\vert _{H^m_w[a,b]}\ < \infty$, we have 
\be
   \left|I_{k}^{[a,b]}(f) - I_{k,N}^{[a,b]}(f) \right|\ \leq
\ \sigma_{m,N}' \,
 \left(\frac{1}{k}\right)^r   h^{m+1-r} \, \left(
  \frac{1}{N}\right)^{m-\rho (r)} \, 
\vert f\vert _{H^m_w[a,b]}\ .
\label{eq:estab} 
\ee

\end{theorem}  
\begin{proof}
From \eqref{eq:trans}, \eqref{eq:transquad} and then Theorem
\ref{thm:filon-cc}, we obtain 
\begin{eqnarray}\left|I_{k}^{[a,b]}(f) - I_{k,N}^{[a,b]}(f) \right|&=&
h \left|I_{\tilde{k}}(\tilde{f}) -
I_{\tilde{k},N}(\tilde{f}) \right|  \notag \\ &\leq&  \sigma_{m,N}'\, 
 h  \left(\frac{1}{\tilde{k}}\right)^r  \left(\frac{1}{N}\right)^{m-\rho(r)} 
\vert \tilde{f}\vert_{H^m_w} \ \notag \\
& = & \sigma_{m,N}' \left(\frac{1}{k}\right)^r   h^{1-r}\left( \frac{1}{N}
\right)^{m-\rho(r)} \vert \tilde{f} \vert_{H^m_w}\ .
\label{eq:inter_step}\end{eqnarray}
Now $\tilde{f}^{(m)}(t) = h^m f^{(m)}\left(c + ht\right)$, and so 
\begin{eqnarray}
\vert \tilde{f}\vert_{H^m_w}^2 
&=& h^{2m} \int_{-1}^{1} \frac{\left|
{f}^{(m)}\left(c + h t\right)\right|^2}
{\sqrt{1-t^2}} \,\, {{\rm d}t} \ = \ h^{2m} \vert f \vert_{H^m_w[a,b]}^2\  
\label{eq:fdash} \end{eqnarray} 
and the result
follows.
\end{proof}

The most important use of this theorem will be for the case when $N$
is fixed and convergence is obtained by letting $h \rightarrow 0$ (as
arises when composite versions of the FCC  rule are used). For this
case we have the following corollary, which is obtained using Theorem 
\ref{th:FCC_ab} with $m = N+1$. 

\begin{corollary}
\label{cor:FCC_ab}
 Let $r \in [0,2]$. For each $N \geq 1$,
there  exists a constant {$c_N = C\sigma_{N+1,N}'  (1/N)^{N+1-\rho(r)} $},  such that  
\be
   \left|I_{k}^{[a,b]}(f) - I_{k,N}^{[a,b]}(f) \right|\ \leq
\ c_N 
\left(\frac{1}{k}\right)^r \,  h^{N+2-r} \, 
\max_{x\in[a,b]} \vert f^{(N+1)} (x) \vert\ 
\ee
when $f \in C^{N+1}[a,b]$.
\end{corollary}  

\section{Composite Clenshaw-Curtis Rules}
\label{sec:Composite}
In this section, we will consider the  computation of
$
I_{k}^{[a,b]}(f)$, 
where  $f$ is allowed to have an algebraic or logarithmic singularity 
in $[a,b]$.  
To control the length of the paper, {we restrict to functions 
$f$ which  are  not continuously differentiable. 
Singularities in higher derivatives can be treated in an analogous way.}


Without loss of generality we set $[a,b] = [0,1]$ and assume 
that the only singularity occurs at the origin. 
The case of a 
finite number of singularities on $[a,b]$ can be 
treated by splitting  $[a,b]$ up into
subintervals, each  with only one singularity at an end point, and
them mapping each interval onto $[0,1]$ in an obvious  affine way.  
Hence,  for $\beta\in(0,1)$ and $m \geq 1$,  we introduce
 \begin{equation}
   ||v||_{m,\beta} := \max\left\{ \sup_{x \in [0,1]} |v(x)| , \sup_{x \in (0,1]}
   \left|x^{(j-\beta)} {v^{(j)}}(x)\right|, \ \ j = 1, ..., m \right\}.
   \label{eq:sing_pos} 
\end{equation}
We denote by $C^m_\beta[0,1]$ the space of all functions $v \in C[0,1]$ such that $||v||_{m,\beta}<\infty$. Similarly, for
$\beta\in(-1,0)$ we define
 \begin{equation} 
   ||v||_{m,\beta} := \max\left\{ \sup_{x \in (0,1]}
   \left|x^{(j-\beta)} {v^{(j)}}(x)\right|, \ \ j = 0, ..., m\right\}
   \label{eq:sing_neg}  \ ,
\end{equation}
and choose  $C^{m}_{\beta}[0,1]$ to be the space of all $v \in C(0,1]$ such that $||v||_{m,\beta}<\infty$. 
Finally, we cover the case of logarithmic singularities via the norm
\begin{equation}
   ||v||_{m,0} := \max\left\{ \sup_{x \in [0,1]} |(|\log x| +1)^{-1}v(x)| ,
\sup_{x \in [0,1]}
   \left|x^{j}{v^{(j)}}(x)\right|, \ \ j = 1, ..., m  \right\} 
   \label{eq:sing_log} 
\end{equation}
and introduce the associated space $C^{m}_{0}[0,1]$. 
Note that $C^{m}_{\beta}[0,1] \subset L^{1}[0,1]$ for all $\beta \in (-1,1)$.

\subsection{The composite algorithm} \label{section:algorithm}

When $f\in C^{m}_{\beta} [0,1]$, for $\beta \in (-1,1)$, our strategy for computing $I_k^{[0,1]}(f)$ 
is to apply the FCC rule in a composite fashion on a mesh graded towards the singularity. 
With the right choice of mesh the error of the
quadrature can then be made to satisfy a uniform error estimate on subintervals and be small overall.
Let us recall the classical graded mesh 
\be \Pi_{M,q} : = \left\{x_{j} := \left(
\frac{j}{M}\right)^{q} : \ j = 0,1,\ldots,M \right\},
\label{eq:mesh}
\ee
where $q \geq 1$ is the grading parameter to be chosen. 
{This  mesh - originally proposed in \cite{Rice:69} -  is well-known to give  
optimal approximation of  functions with singularities by fixed order 
piecewise polynomials.  An application to  quadrature was given 
in  \cite{GrMe:89}.   This paper
contains  an extension of these  results to the computation   
of  oscillatory integrals with singularities.}  
{Writing }
\[
I^{[0,1]}_k(f) = I^{[x_0,x_1]}_k(f) + \sum_{j = 2}^{M}I^{[x_{j-1},x_j]}_k(f),
\]
we approximate each term in the sum on the right-hand side by applying the FCC rule as 
defined in (\ref{eq:transquad}). The strategy for approximating the first term on the 
right-hand side depends on whether $\beta \leq 0$ or $\beta >0$. Precisely we define the approximation 
\begin{equation}
\widetilde{I}_k^{[x_0,x_1]}(f) :=\left\{  \begin{array}{ll} 
{I^\igg{[x_{0},x_1]}_{k,1}(f)},
 & \text{if} \ \  
\beta \in (0,1),\\ 0, & \text{if} \ \  \beta \in (-1,0].\end{array} \right.  
\  \label{eq:first_approx_def}
\end{equation}
Note that for $\beta\in(0,1)$,  
\[
I^{[x_{0},x_1]}_{k,1}(f)=\left\{
\begin{array}{ll}
\displaystyle\int_{x_0}^{x_1}\left(Q_1^{[x_0,x_1]}f\right) (x) \exp({\rm i}kx)\,{\rm d}x, &\text{if }\ {x_1k\geq1}\\
\displaystyle\int_{x_0}^{x_1}\left(Q_1^{[x_0,x_1]}\left(f \exp({\rm i}k\,\cdot\,)\right)\right)(x)\,{\rm d}x, &\text{if }  \ {x_1k<1}
\end{array}
\right.
\]where $Q_1^{[x_0,x_1]}f$ is the linear function interpolating $(x_0,
f(x_0))$ and $(x_1, f(x_1))$. 
\igg{(To obtain this formula, recall that
from  \eqref{eq:trans}, $I_{k,1}^{[x_0,x_1]}(f) = h \exp(ikc)
I_{\tilde{k},1}(\tilde{f})$ where $\tilde{k} = k x_1/2$, 
and recall
\eqref{eq:therule} and \eqref{eq:therule:02}).}
The composite quadrature rule is
\be
I^{[0,1]}_{k, N,M,q}(f) \ := \ \tilde{I}^{[x_0,x_1]}_k(f) + \sum_{j = 2}^{M}I^{[x_{j-1},x_j]}_{k,N}(f).
\ee
The corresponding error \igg{may then be bounded by
\begin{equation}
E_{k,N,M,q}(f) \ : = \  \left| I_k^{[0,1]}(f) - I^{[0,1]}_{k, N,M,q}(f)\right|
\ \leq \ |\tilde{e}_1| + \sum_{j = 2}^{M}|e_j| ,
\label{eq:basic_error_bd}\end{equation}}
where
\begin{eqnarray}
\tilde{e}_1 &=& I^{[x_0,x_1]}_k(f) - \tilde{I}^{[x_0,x_1]}_k(f) , \ \ \ \ \ \ \ \text{and} \label{eq:e1}\\
e_j &=& I^{[x_{j-1},x_j]}_k(f) - {I}^{[x_{j-1},x_j]}_{k,N}(f), \ \ \ \ \ \  \text{for} \ \ j = 2,...,M. \label{eq:ej}
\end{eqnarray}
In the following two sections, 
we derive results which will help us estimate 
 $|\tilde{e}_1|$. These are subsequently   used to estimate the total error $E_{k,N,M,q}(f)$ in 
 Theorem
\ref{thm:beta_all}. 

\subsection{Estimates on the size of the integrals}\label{section:estimates_of_int}
In the following two lemmas, we analyse the integrals  
$I^{[0,\varepsilon]}_k (f)$, where $f\in C^1_\beta[0,1]$, making  explicit  
the rate of decay as both $\varepsilon \rightarrow 0$ and $k \rightarrow \infty$.


\begin{lemma}
\label{lemma:betaNeg} 
For any $\beta\in(-1,0)$ and any $f\in C^1_\beta[0,1]$ there exists
$C_\beta >0$
such that for  $\varepsilon\in(0,1]$ we
have
\begin{equation}\label{eq:01:lemma:betaNeg}
 |I_{k}^{[0,\varepsilon]} (f) |\ \le\  C_\beta 
\varepsilon^{1+\beta-s}\bigg(\frac{1}k\bigg)^s\|f\|_{1,\beta}
\end{equation}
where   $s\in[0,1+\beta]$.
Furthermore,
for any $f\in C^1_{0}[0,1]$ there exists $C_0>0$ such that for $\varepsilon\in(0,1]$, we have
\begin{equation}
\label{lemma:betalog}
|I_{k}^{[0,\varepsilon]} (f) |\ \le \ C_0
\big(\varepsilon  +\varepsilon|\log\varepsilon| \big)^{1-s}\bigg(\frac{1+ \log
k}k\bigg)^{s} \|f\|_ { 1 , 0 }
\end{equation}
where  $s\in[0,1]$ .
\end{lemma}

\begin{proof}
We will prove the lemma for the case when $\beta \in (-1,0)$. The case
when $\beta =0$ follows similarly. 
First note that for all $\varepsilon\in(0,1]$, we have
\begin{equation}\label{eq:02:lemma:betaNeg}
  |I_{k}^{[0,\varepsilon]} (f) |\ \le\  \bigg[\int_0^\varepsilon
x^{\beta}\,{\rm d}x\bigg]\|f\|_{0,\beta}\ =
\ \frac{1}{1+\beta}\varepsilon^{1+\beta}\|f\|_{0,\beta}
\ \le\ \frac{1}{1-|\beta|} \varepsilon^{1+\beta}\|f\|_{ 1,\beta}.
\end{equation}
We show now 
{\begin{equation}\label{eq:03:lemma:betaNeg}
  |I_{k}^{[0,\varepsilon]} (f) |\ \le \ 
\bigg[\frac{1}{1-|\beta|}\bigg]\bigg[\frac{2+|\beta|}{|\beta|}\bigg]
\bigg(\frac{1} k\bigg)^{1+\beta} \|f\|_{1,\beta}.
\end{equation}
and the result follows by interpolation of \eqref{eq:02:lemma:betaNeg} and
\eqref{eq:03:lemma:betaNeg}. 

To obtain  \eqref{eq:03:lemma:betaNeg}, note first that it follows trivially from
\eqref{eq:02:lemma:betaNeg} if $\varepsilon k\leq 1$. Therefore, let us assume that
$\varepsilon>1/k$, and write  
\begin{equation}\label{eq:03b:lemma:betaNeg}
 I_{k}^{[0,\varepsilon]} (f) =I_{k}^{[0,1/k]} (f) +I_{k}^{[1/k,\varepsilon]}
(f). 
\end{equation}
Now again from \eqref{eq:02:lemma:betaNeg}  
\begin{equation}\label{eq:04:lemma:betaNeg}
 |I_{k}^{[0,1/k]}(f)|\le \frac{1}{1-|\beta|}
\bigg(\frac1k\bigg)^{1+\beta}\|f\|_{
0,\beta}.
\end{equation}
On the other hand,  integration by parts yields
\[
 I_{k}^{[1/k,\varepsilon]}(f)=\frac{1}{{\rm i}k}\Big[f(x)\exp({\rm
i}kx)\Big]_{x=1/k}^{x=\varepsilon}-\frac{1}{{\rm i}k}\int_{1/k}^{\varepsilon}
f'(x)\exp({\rm i}kx)\,{\rm d}x.
\]
Thus
\begin{eqnarray*}
| I_{k}^{[1/k,\varepsilon]}(f)|&\le& \frac{1}{k}
\bigg[|f(\varepsilon)|+|f(1/k)| +\int_{1/k}^\varepsilon
|f'(x)|\,{\rm d}x \bigg]\ .
\end{eqnarray*}
Now 
for any $x>0$, \, 
$\vert f(x)\vert \leq x^\beta \Vert f \Vert_{1,\beta}$\,   and also  $$\int_{1/k}^\varepsilon \vert f'(x)\vert \rd x \ \leq  \ \int_{1/k}^\varepsilon x^{\beta - 1} \rd x \, \Vert f 
\Vert_{1,\beta}  \ \leq \ \frac{1}{\vert \beta \vert} \left[\varepsilon^\beta + \left(\frac{1}{k}\right)^\beta \right] \Vert f \Vert_{1,\beta}\ .
$$
Thus 
\begin{eqnarray*}
| I_{k}^{[1/k,\varepsilon]}(f)| &\le& \frac{1}k\Big[ 1+
 \frac{1}{|\beta|}\Big]\Big[ \varepsilon^{ \beta}+
\Big(\frac{1}k\Big)^{\beta}\Big] \|f\|_{1,\beta}, 
\end{eqnarray*}
and, since $\varepsilon^\beta<(1/k)^\beta$ we obtain
\begin{equation}\label{eq:05:lemma:betaNeg}
 |
I_{k}^{[1/k,\varepsilon]}(f)|\ \le\ {2}\Big[\frac{1+|\beta|}{|\beta|}\Big]\Big(\frac{1}
{k}\Big)^{1+\beta}\|f\|_{1,\beta}.
\end{equation}
Substituting \eqref{eq:05:lemma:betaNeg} and \eqref{eq:04:lemma:betaNeg} into
 \eqref{eq:03b:lemma:betaNeg} we obtain 
$$
I_k^{[0,\varepsilon]}(f) \ \leq \ \left[
\frac{|\beta| + 2 (1 - |\beta|^2)}{(1 - \vert \beta\vert)|\beta|}
\right]
\Big(\frac{1}
{k}\Big)^{1+\beta}\|f\|_{1,\beta}
$$
thus proving 
\eqref{eq:03:lemma:betaNeg}.}
\end{proof}

\begin{lemma}\label{lemma:betaPos} For any $\beta\in(0,1)$ and for any $f\in
C_{\beta}^2[0,1]$ there exists $C_{\beta}>0$
such that for $\varepsilon\in(0,1]$ and $s\in[0,1]$ we have
\begin{equation}
 \label{eq:01:lemma:betaPos}
 |I_{k}^{[0,\varepsilon]} (f) |\le C_\beta
\varepsilon^{1-s} \bigg(\frac{1}k\bigg)^{s}\|f\|_{1,\beta}.
\end{equation}
Moreover, if $s\in[1,1+\beta]$,
 \begin{eqnarray}
 |I_{k}^{[0,\varepsilon]} (f) |&\le&
\frac{1}{k}\bigg[|f(0)|+|f(\varepsilon)|\bigg]+
C_\beta \varepsilon^{1+\beta-s} \bigg(\frac{1}k\bigg)^{s}\|f'\|_{1,\beta-1}\nonumber \\
&\le& \frac{2}{k} \|f\|_{0,\beta}+
C_\beta \varepsilon^{1+\beta-s} \bigg(\frac{1}k\bigg)^{s}\|f\|_{2,\beta}. \label{eq:02:lemma:betaPos}
\end{eqnarray}
\end{lemma}
\begin{proof}
First note that for $\varepsilon \, {\in}\,  (0,1]$, 
 \begin{equation}
 \label{eq:03:lemma:betaPos}
  |I_{k}^{[0,\varepsilon]}(f) |\le 
\varepsilon \|f\|_{0,\beta}.
\end{equation}
On the other hand, integration by parts  yields
 \begin{eqnarray}
\left| I_{k}^{[0,\varepsilon]} (f)\right| &\le& \frac{1}{k}\bigg[|f(0)|+|f(\varepsilon)|\bigg]+
\frac{1}k\bigg|\int_0^\varepsilon f'(x)\exp({\rm i}kx)\,{\rm d}x\bigg| \label{eq:inter} \\
&\le & \frac{2}{k} \|f\|_{0,\beta} + \frac{1}{k} \left| \int_{0}^{\varepsilon} f'(x) \exp({{\rm i}}kx) \, {\rm d}x \right| . \label{eq:05:lemma:betaPos}
\end{eqnarray}
{Also},  we have
\begin{equation}
\label{eq:06:lemma:betaPos}
\left| \int_{0}^{\varepsilon} f'(x) \exp({{\rm i}}kx) \, {\rm d}x \right| \leq \bigg[\int_{0}^{\varepsilon} x^{\beta -1} 
\, {\rm d}x\bigg] \|f\|_{1,\beta} = \frac{1}{\beta} \varepsilon ^{\beta}  \|f\|_{1,\beta}
\end{equation}
and substitution of (\ref{eq:06:lemma:betaPos}) into (\ref{eq:05:lemma:betaPos}) yields
\[ 
\left| I_{k}^{[0,\varepsilon]} (f)\right| \leq \left(2+\frac{1}{\beta}\right) \frac{1}{k} \|f\|_{1,\beta}.
\]
Interpolation of this with (\ref{eq:03:lemma:betaPos}) yields (\ref{eq:01:lemma:betaPos}). 

To obtain (\ref{eq:02:lemma:betaPos}), we also note $f' \in C^{1}_{\beta - 1}[0,1]$ so by Lemma \ref{lemma:betaNeg}, 
we also have 
\igg{\begin{eqnarray*}
\left| \int_{0}^{\varepsilon} f'(x) \exp({{\rm i}}kx) \, {\rm d}x \right| &\leq& C_{\beta - 1} \varepsilon ^{\beta - {s'}} 
\left( \frac{1}{k} \right)^{{s'}} \|f'\|_{1,\beta - 1} \ \leq \  C_{\beta - 1} \varepsilon ^{\beta - s'}
\left( \frac{1}{k} \right)^{s'} \|f\|_{2,\beta }
\end{eqnarray*}} 
for all $s' \in [0, \beta]$. Substituting this into {(\ref{eq:inter})} and putting $s = 1+s'$, we obtain the result.
\end{proof}


In the next lemma 
we shall verify the sharpness of the estimates in   Lemmas \ref{lemma:betaNeg}\--\ref{lemma:betaPos}  as  $k \rightarrow \infty$.  {This result concerns}  the family of functions 
\begin{equation}\label{eq:deffbeta}
f_{\beta}(x) =  \left\{\begin{array}{ll}
x^{\beta}\ , &  \beta \in (-1,0) \cup (0,1),\\
\log x \ , & \beta =0.
 \end{array} \right.
\end{equation}

\begin{lemma}
\label{remark:exact_integrals}
For all $\beta \in (-1,1)$, there exists a constant $A_\beta> 0$ such
that, \igg{for all $k$ sufficiently large},  
\begin{eqnarray}
k^{\min\{1+\beta, 1\}} |I_{k}^{[0,1]}(f_{\beta})| &\geq& A_\beta,  \ \ \ \ \ \ \text{when} \ \ \beta \not = 0  \label{eq:a1}\\
\frac{k}{\log k} |I_{k}^{[0,1]}(f_{0})| &\geq& A_0. \label{eq:a3}
\end{eqnarray}
\end{lemma}

\begin{proof}
Let us consider first 
$\beta \in (-1,0) \cup (0,1)$.
Using  \cite[(3.761.1),(3.761.6)]{{GradRyz:1994}}, we obtain 
\begin{equation} 
I_{k}^{[0,1]}(f_{\beta}) = \frac{1}{1+\beta} \ {}_1 F_1 \left(1+\beta,
2+\beta, {{\rm i}}k \right),
\label{eq:exact_kummers} \end{equation}
where  $ {}_1 F_1 (a,b,z)$ denotes  the confluent hypergeometric 
function (also called Kummer's function and  denoted $M(a,b,z)$ in 
\cite[Section 13]{AbrSt}). At large values of $|z|$, with $-\pi/2<\arg
z< 3\pi/2$, for fixed $a$ and $b$, the function
$ {}_1 F_1 (a,b,z)$ has the following asymptotics, see \cite[(13.5.1)]{AbrSt}:
\[
  {}_1 F_1 (a,b,z) \ = \ {\Gamma(b)} \left( 
\frac{e^{{{\rm i}}\pi a} z^{-a}}{\Gamma(b-a)} \ + \  
\frac{ e^{z} z^{a-b}}{\Gamma(a)} \right)
\left(1 +O(|z|^{-1})\right) .
\]
Then, from (\ref{eq:exact_kummers}) and since $\Gamma(1+z) = z \Gamma(z)$, we obtain,
\begin{eqnarray*}
I_{k}^{[0,1]}(f_{\beta}) 
&=&  \left(\Gamma(1+\beta) e^{{{\rm i}}\pi(1+\beta)}  ({{\rm i}}k)^{-1-\beta}
+ { e^{{{\rm i}}k}} ({{\rm i}}k)^{-1} \right) \left(1+
O(k^{-1}) \right). 
\end{eqnarray*}
Therefore, for  $k$ sufficiently large, we have, for $\beta \in (-1,0)$, 
\begin{eqnarray*}
\vert I_{k}^{[0,1]}(f_{\beta})\vert  
&\geq & \frac{1}{2}  \left(\Gamma(1+\beta)  k^{-1-\beta}
-  k^{-1} \right) 
\ \geq \ \frac{1}{4} \Gamma(1+\beta)  k^{-1-\beta} ,
\end{eqnarray*}
and for $\beta \in (0,1)$, 
\begin{eqnarray*}
\vert I_{k}^{[0,1]}(f_{\beta})\vert  
&\geq & \frac{1}{2}  \left( k^{-1} - \Gamma(1+\beta)  k^{-1-\beta}
  \right) 
\ \geq \ \frac{1}{4} k^{-1} .
\end{eqnarray*}
These prove the estimate (\ref{eq:a1}).

To verify (\ref{eq:a3}), we use the formulae
\cite[(4.381.1), (4.381.2)]{GradRyz:1994} to obtain 
\[
 I^{[0,1]}_k(f_0)=-\frac{1}{k} \left( \frac{\pi}{2} + {{\rm i}} \gamma + {{\rm i}} \log k\right) 
 + \frac{{\rm i}}{k} \big({\rm ci}(k)+{{\rm i}}{\rm
si}(k)\big)
\]
where ${\rm si}$ and ${\rm ci}$ are the sine and cosine integral functions:
\[{\rm si}(x) := -\frac{\pi}{2} + \int_{0}^{x} \frac{\sin t}{t} \,{{\rm d}t}, \  \ \ \ \ \ \ 
{\rm ci}(x) := \gamma + \log x+ \int_{0}^{x} \frac{\cos t -1 }{t} \,{{\rm d}t},  \]
 and
$\gamma\approx 0.5772$ is the Euler-Macheroti constant. (Note that, in the notation of \cite{AbrSt}, ${\rm si}(x) = -\pi/2 + {\rm Si}(x)$ and ${\rm ci}(x) = 
{\rm Ci}(x)$.) 
Then using the asymptotics for large arguments of ${\rm Si}$ and ${\rm Ci}$ in 
\cite[(5.2.34), (5.2.35)]{AbrSt} and 
\cite[(5.2.8), (5.2.9)]{AbrSt}, we deduce that 
\[
 \lim_{k\to \infty} {\rm si}(k) \ = \mathcal{O}(1/k),\quad
 \lim_{k\to \infty} {\rm ci}(k)\ = \ \mathcal{O}(1/k) .
\]
Thus 
\[
I_{k}^{[0,1]}(f_0) = - \frac{1}{k} \left(\frac{\pi}{2} +{{\rm i}} \gamma+ {{\rm i}} \log k\right) + O\left(\frac{1}{k^2}\right).
\] 
Thus, for all $k$ sufficiently large
\[
| I_{k}^{[0,1]}(f_0)| \ \geq \ \frac{\log k}{2k} ,
\]
proving the estimate (\ref{eq:a3}).
\end{proof}

\subsection{The total error for the composite Filon-Clenshaw-Curtis method} \label{section:total_error}

In Theorem \ref{thm:beta_all} {below we  use \eqref{eq:basic_error_bd}} to 
estimate the total error of the 
composite FCC rule. 
The first contribution 
$|\tilde{e}_1|$ is estimated either by a direct application of  
Lemma \ref{lemma:betaNeg} (when $\beta \in (-1,0]$),  
 or via an integration by parts argument (when 
 $\beta \in (0,1)$).   
{This is done in Lemma \ref{lem:e1} below, but first} the remaining sum on the right-hand side   
is estimated in the following lemma. Since the proof uses   
 fairly classical graded mesh arguments we shall be brief.  
\begin{lemma} \label{lemma:sum_ej}
Let $f \in C^{N+1}_{\beta}[0,1]$, $\beta \in (-1,1)$, let $r \geq 0$ 
and choose 
\begin{equation}\label{defq}
q \ > \ (N+1-r)/(\beta+1-r), \ \ \ \ \ \text{for} \ \ r <1+\beta.
\end{equation}
Then there exists a constant $C$ which depends on $N$, $\beta$ and $q$ such that 
\be \label{eq:result_positive_lemma}
  \sum_{j = 2}^{M} \left|e_j\right| \ \leq\ C
   \, \left(\frac{1}{k}\right)^r \, \left(\frac{1}{M}\right)^{N+1-r} \, \left\|f\right\|_{N+1,\beta},
\ee
\end{lemma}

\begin{proof}
In the proof we let $C$ denote a generic constant which may depend on 
$N$, $\beta $ and $q$. 
By Corollary \ref{cor:FCC_ab}, and denoting
 $h_j=(x_j-x_{j-1})/2$, we have 
\begin{eqnarray}
\sum_{j = 2}^{M} \left|e_j\right| &\ \leq \ &
C   \left(\frac{1}{k}\right)^r \, 
\sum_{j = 2}^{M} h_{j}^{N+2-r } \max_{x \in [x_{j-1}, x_{j}]}
\left|f^{(N+1)}(x)\right| \notag \\ &\ \leq\ &C \left(
  \frac{1}{k}\right)^r 
\left\{  \sum_{j = 2}^{M} h_{j}^{N+2-r}
x_{j-1}^{\beta - N-1} \right\} \, \left\|f\right\|_{N+1,\beta} \ . \label{eq:inter_proof_sing}
\end{eqnarray}
A simple  application of the mean-value theorem shows that  
$$h_j \ \leq \ C \frac{1}{M} \left(\frac{j-1}{M}\right)^{q-1} , \quad \text{for}  \quad j = 2,\ldots,M, $$  and hence 
$$
h_j^{N+2-r} x_{j-1}^{\beta-N-1} \ \leq \ C \left(
  \frac{1}{M}\right)^{N+2-r} 
 \left( \frac{j-1}{M} \right)^{\alpha} 
$$
where $\alpha = q(\beta + 1 -r) - (N+2 -r)$, and so 
\be
\sum_{j = 2}^{M} \left|e_j\right| \ \leq\  C
\left(\frac{1}{k}\right)^r \left( \frac{1}{M}\right)^{N+1-r}    
\, \left\{ \sum_{j=2}^M \frac{1}{M} \left( \frac{j-1}{M} \right)^{\alpha} \right\} \,  
\left\|f\right\|_{N+1,\beta} ,
\label{eq:summodej} \ee
 The result follows since the  factor in braces in 
\eqref{eq:summodej} is a Riemann sum for the integral 
$\int_0^1 x^\alpha \, {\rm d}x$ - which is finite,   
since the hypothesis of the lemma ensures that  $\alpha > -1$. 
\end{proof}




\begin{lemma}
\label{lem:e1}
Under the same hypothesis as Lemma \ref{lemma:sum_ej},  
there exists a constant $C$ which depends on  $\beta$ and $q$ such that, {for $N \geq 1$} 
\begin{equation*} 
   \vert \widetilde{e}_1 \vert  \ \leq  \   C  \left(\frac{1}{k}\right)^r \left(\frac{1}{M}\right)^{N+1-r}
\left\{\begin{array}{ll}    
\, \displaystyle{ 
 \,  \, 
\left\|f\right\|_{\ig{2,\beta}}} \ ,  & \text{when} \quad 
\displaystyle{\beta \in (-1,0)\cup (0,1) , } \\
\\ 
   \, \displaystyle{ \left({1+ \log k}\right)^r \,  (\log M)^{1-r}\, 
\left\|f\right\|_{\ig{2,0}} \ } , &   \text{when} \ \displaystyle{ \beta=0}.
\end{array}\right. 
\end{equation*}
\end{lemma}
\begin{proof}
{Throughout we use the fact that $x_0 = 0$. 
Consider first  $\beta\in(0,1)$ and note  that 
\begin{eqnarray}
\vert (Q_1^{[0,x_1]}f)' \vert \ &=& \ \left\vert \frac{f(x_1) - f(0)}{x_1}\right\vert  \ =\ \left\vert  \frac{1}{x_1}\int_0^{x_1} f'(x) \rd x \right\vert \ \leq 
\  \frac{1}{x_1}\int_0^{x_1} \vert f'(x)\vert  \rd x \nonumber
\\
&\le&    \frac{1}{x_1} \,  \bigg[\int_0^{x_1} x^{\beta-1}\,{\rm d}x\bigg]\|f\|_{1,\beta}  
\ = \   {\frac{1}{\beta}}  x_1^{\beta-1} \|f\|_{1,\beta}  \ . \label{eq:Q1prime}
\end{eqnarray}
Moreover,   for any  $t\in[0,x_1]$ and  any   ${{f}}\in C_\beta^ 1[0,1]$, we have 
\begin{eqnarray}
|f(t)-Q_{1}^{[0,x_1]}f(t)|
&=&\bigg|\int_0^{t}(f-Q_{1}^{[0,x_1]}f)'(x)|\,{\rm d}x\bigg|\nonumber\\
&\le& \int_0^{x_1}  |f'(x)|\,{\rm d}x+x_1 \big|\big(Q_{1}^{[0,x_1]}f\big)'\big|
\leq \ 
\frac{2}{\beta}  x_1^{\beta} \|f\|_{1,\beta}\ .
\label{eq:02:thm:beta_all} 
\end{eqnarray}
Then,  
from 
\eqref{eq:first_approx_def}, we have from \eqref{eq:02:thm:beta_all}, {when \ig{$f \in C^{1}_\beta[0,1]$},} 
\begin{eqnarray}
 |\widetilde{e}_1 |&=& \bigg|\int_0^{x_1} \Big(f(x)-Q_{1}^{[0,x_1]}f(x)\Big)\exp({\rm i}kx)\, {\rm d}x\bigg| 
 \ {\le \  \frac{2}{\beta}  \, x_1^{1+\beta}\,  \|f\|_{\ig{1,\beta}}}
 \ . 
 \label{eq:03:thm:beta_all} 
\end{eqnarray}
On the other hand, integrating  the formula for $\widetilde{e}_1$ 
by parts, we obtain
\begin{equation}
\widetilde{e}_1 \ = \ - \frac{1}{\ri k}  
\int_0^{x_1} \left(f(x) - (Q^{[0,x_1]}_1 f)(x)\right)'\exp({{\rm i}}kx)  \rd x 
\ . \label{eq:intbyparts}
\end{equation}  
Since   $f'\in C_{\beta-1}^N$,  to treat 
the first term on the right-hand side of   
\eqref{eq:intbyparts} we can use     
Lemma \ref{lemma:betaNeg} with  $\beta$ replaced by $\beta-1$ 
and $s$ chosen to be  $\beta$, thus obtaining   
\begin{equation}\label{eq:1stterm}
\left\vert\int_0^{x_1} f'(x) \exp({{\rm i}}kx)  \rd x\right\vert\ \leq \ C_\beta  \left(\frac{1}{k}\right)^{\beta} \Vert f' \Vert_{1,\beta-1} \leq \ C_\beta  
\left(\frac{1}{k}\right)^{\beta} \Vert f \Vert_{2,\beta}\ . 
\end{equation}
Moreover, to treat the second term in \eqref{eq:intbyparts}, since 
$(Q^{[0,x_1]}_1 f)'$ is constant, we have by \eqref{eq:Q1prime},  
\begin{equation}\label{eq:2ndterm}
\left\vert \int_0^{x_1} (Q^{[0,x_1]}_1 f)' \exp({{\rm i}}kx)  \rd x \right\vert
\ = \   \left\vert (Q^{[0,x_1]}_1 f)'\right\vert \,  \left\vert \int_0^{x_1}  \exp({{\rm i}}kx)  \rd x\right\vert\ \leq \ \frac{x_1^{\beta-1}}{\beta k}\, \Vert f \Vert_{1,\beta} \ . 
\end{equation}
Hence combining \eqref{eq:1stterm} and \eqref{eq:2ndterm} with  \eqref{eq:intbyparts}, we obtain 
\begin{equation}
\label{eq:04:thm:beta_all} 
| \widetilde{e}_1| \ \le \ \igg{C_{\beta}'} \left[
\left(\frac{1}{k} \right)^{1+\beta}   +  \frac{x_1^{\beta-1}}{k^2} \right] 
\|f\|_{2,\beta}\ \leq  \ \igg{C_\beta'} \left(\frac{1}{k}\right)^{1+ \beta} \left[ 1 + (x_1 k)^{\beta-1}\right] \Vert f \Vert_{\ig{2},\beta} \ . 
\end{equation}

Hence if \igg{$x_1 k\geq 1$}  we can  interpolate  
\eqref{eq:04:thm:beta_all} and \eqref{eq:03:thm:beta_all}, 
to deduce that 
\be \label{eq:06:thm:beta_all}
 \vert \widetilde{e}_1 \vert \ \le\  \igg{C_\beta''} \bigg(\frac1{k}\bigg)^{r} x_1^{1+\beta-r}  \|f\|_{N+1,\beta}\le \igg{C_\beta''} \bigg(\frac1{k}\bigg)^{r} \bigg(\frac{1}{M}\bigg)^{N+1-r} \|f\|_{\ig{2},\beta}\ , 
\ee
for any $r\in[0,1+\beta]$. 

On the other hand, 
if \igg{$kx_1<1$}, and defining $f_k(x):=f(x)\exp({\rm i}kx)$,  \eqref{eq:02:thm:beta_all} yields 
\begin{eqnarray}
|\widetilde{e}_1 | \ & = & \  \bigg|\int_0^{x_1} \big(f_k(x)-Q_{1}^{[0,x_1]}f_k(x)\big) 
 {\rm d}x\bigg|\  \le  \  \igg{\frac{2}{\beta}}  x_1^{1+\beta} \|f_k\|_{1,\beta}\nonumber \\
&  \le &  \ 
  \igg{\frac{2}{\beta}}\bigg(\frac{1}{k}\bigg)^{r} x_1^{1+\beta -r}  \|f\|_{\ig{1},\beta} 
\ \le \ \igg{\frac{2}{\beta}} \bigg(\frac1{k}\bigg)^{r} \bigg(\frac{1}{M}\bigg)^{N+1-r} \|f\|_{\ig{1},\beta}. \label{eq:05:thm:beta_all} 
\end{eqnarray} 
and   \eqref{eq:06:thm:beta_all}, \eqref{eq:05:thm:beta_all} prove 
the result for  $\beta\in(0,1)$\ . 
    }

For $\beta \in (-1,0]$ note that 
the integral over the first subinterval is approximated by zero,
and so the result follows readily from Lemma \ref{lemma:betaNeg} (with $\varepsilon = x_1$ and $s = r$).
\end{proof}

The proof of the following result now follows directly from Lemmas \ref{lemma:sum_ej} and  \ref{lem:e1}. 
\begin{theorem}
\label{thm:beta_all}
Under the same hypothesis as Lemma \ref{lemma:sum_ej},
there exists a constant $C$ which depends on $N$,  $\beta$ and $q$ such that 
\begin{equation*} 
   E_{\igg{k,N,M,q}}(f) \ \leq  \   C  \left(\frac{1}{k}\right)^r \left(\frac{1}{M}\right)^{N+1-r}
\left\{\begin{array}{ll}    
\, \displaystyle{ 
 \,  \, 
\left\|f\right\|_{N+1,\beta}} \ ,  & \hspace{-1in} \text{when} \quad 
\displaystyle{\beta \in (-1,0)\cup (0,1) , } \\
\\ 
   \, \displaystyle{ \left({1+ \log k}\right)^r \,  (\log M)^{1-r}\, \left\|f\right\|_{N+1,0} \ } , &   \text{when} \ \displaystyle{ \beta=0}.
\end{array}\right. 
\end{equation*}
\end{theorem}

\section{Nonlinear oscillators }
\label{sec:Composite_Stationary}
In this section we return to the integral $I_k^{[a,b]}(f,g)$ 
defined in (\ref{eq:theproblem}), 
and consider a general nonlinear $g$. We will assume for simplicity that 
$g \in C^{\infty}[a,b]$. For less smooth $g$ the arguments will be 
analogous but the exposition would be more technical. 
Our  methods will be based on the  change of variable 
\begin{equation}
\tau = g(x).   \label{eq:transform}
\end{equation}  
If $g$ has no stationary points (i.e. $g'$ does not vanish), then 
$g^{-1} \in C^{\infty}[a,b]$,   $(g^{-1})' =  1/(g' \circ g^{-1})$  
and  
\begin{equation}
I_{k}^{[a,b]}(f,g) \ = \ I_{k}^{[g(a),g(b)]}(F), \quad \text{with} \quad 
F =   (f \circ 
g^{-1}) \vert (g^{-1})'\vert   =  (f \circ 
g^{-1}) \vert  (g' \circ g^{-1}) \vert^{-1}  .
\label{eq:capitalFdef}
\end{equation}
Now, assuming also   that  $f \in C^{\infty}[a,b]$,
then  (\ref{eq:capitalFdef}) can
be computed using the FCC rules described in \S \ref{sec:CleCur}, 
with the additional cost being the
evaluation of the inverse function $g^{-1}$ at the quadrature points. 
Moreover if  $f$ has singularities 
then these induce singularities in $F$ and the composite FCC rules in
\S \ref{sec:Composite}  could be used instead. (\igg{Here we assume 
implicitly that $g(a) > g(b)$. If $g(b) > g(a)$ then the integral can
be transferred to one over the interval $[g(b), g(a)]$ by a simple
affine change of variables. We make similar implicit assumptions below.})


If now $g$ has a 
stationary point at one or more $\xi \in [a,b]$,  the transformation  (\ref{eq:transform})
may still be applied,  but    
  singularities appear in $F$ at the points $g(\xi)$. 
To describe these, we may,  
without loss of generality,
consider a 
single stationary point $\xi \in [a,b]$ of order $n\geq 1$ with property 
\begin{equation}
g'(\xi) = g''(\xi) = \ldots = g^{(n)}(\xi) \, = \, 0, \ \    
g^{(n+1)}(\xi)  >   0  \ \ \text{and} \ \   
g'(x) \not =  0,   \ \text{for}\     x\in [a,b]\setminus \{\xi\}. \label{eq:statpointorden}
\end{equation}
Then $g$ is monotone on each of the intervals $[a,\xi)$ and $(\xi,b]$.  
The change of variables (\ref{eq:transform}) can be applied on each interval 
separately to obtain 
\begin{eqnarray}
I^{[a,b]}_{k}(f,g) 
&=& \left(\int_{g(a)}^{g(\xi )} + \int_{g(\xi)}^{g(b)}\right) F(\tau) \exp({{\rm i}}k
\tau) {\rm d}\tau 
\label{eq:sing_int_1d}
\end{eqnarray}
with $F$ as in \eqref{eq:capitalFdef}.  (One of the integrals in \eqref{eq:sing_int_1d} is interpreted as void if $\xi = a$ or $b$.) The regularity of the resulting function  
$F$ is summarised in the following theorem.  Here and throughout the rest of 
this section we use the convenient notation $\alpha = 1/(n+1)$. 
\begin{theorem} \label{thm:derivofF}
Assume that $f$ and $g$ are in $C^{\infty}[a,b]$ and that 
$g$ has a single stationary point $\xi$ \igg{of order $n$} as in \eqref{eq:statpointorden}. Then for each $p \in \mathbb{N}_0 := \mathbb{N}\cup \{ 0\}$, there exists $C_p >0$ such that
\be
\left|F^{(p)} (\tau)\right| \ \leq \ C_{p} \left|\tau - g(\xi)
\right|^{\alpha- p-1}, \ \ \ \ \ \ \tau \in [g(a), g(\xi))\cup (g(\xi), g(b)] \ .
\label{eq:F_1_deriv_est}
\ee
\end{theorem}

The proof of Theorem \ref{thm:derivofF} requires Lemma \ref{lemma:psi_inv_est}
(below) and  both these  results require the 
Fa\`{a} di Bruno formula  for 
  the derivatives of the 
composition 
of two  univariate functions.  
For $p \in \mathbb{N}_0$, let $f^{(p)}$ denote the $p$th derivative of any 
sufficiently  differentiable  function $f$.  Then,  if  $\phi$,   $\psi$ are 
suitably smooth functions 
 and    
 the composition $\phi\circ \psi$ is well-defined,  we have the formula 
\be
(\phi\circ \psi)^{(p)} \ = \displaystyle{\sum} \, \frac{p!}{m!} \left(\phi^{(\vert m \vert)}\circ\psi\right)\,  \left(
\prod_{j = 1}^{p} \left(\psi^{(j)}\right)^{m_j}
\right)  , \label{eq:A}\ee
where  the sum in \eqref{eq:A} is over all multiindices ${m} \in (\mathbb{N}_{0})^p$ 
which satisfy 
$
m_1+2m_2+\ldots+pm_p = p  
$. Moreover   $\vert m \vert = m_1 + m_2 + \ldots + m_p$\,  and \,  $m! = m_1! m_2! \ldots m_p!$.  A suitable reference is (\cite[Theorem 2]{Ro:80}). 
\begin{lemma} \label{lemma:psi_inv_est}
Assume $g \in C^\infty[a,b]$ and that $g$ has a single stationary point 
$\xi$ \igg{of order $n$} as  in \eqref{eq:statpointorden}. Then, 
for all $p \in \mathbb{N}$, there exists a constant $C_p>0$  such that 
\be
\left| \left( g^{-1}\right)^{(p)}(\tau)\right| \ \leq \ C_p \left|\tau -
g(\xi) \right|^{\alpha -p} . \label{eq:est_on_inv_psi}
\ee
\end{lemma}
\begin{proof}
\igg{Without loss of generality} we assume $\xi < b$, \igg{that $g$ is
increasing in $[\xi,b]$} and we  
consider the case $\tau \in (\xi, b]$ 
only. (The case $ \tau \in [a,\xi)$
is completely analogous.)
By Taylor's theorem with integral remainder and \eqref{eq:statpointorden},  
\be \label{eq:intdeq} g(x) \ = \ g(\xi) + (x-\xi)g'(\xi)+\ldots+
\frac{(x-\xi)^{n}}{n!}g^{(n)}(\xi)+ R_{\xi}(x) \ = \ g(\xi) + R_{\xi}(x) ,
\ee
for all $x\in (\xi, b]$, where 
\be R_{\xi}(x) = \frac{1}{n!}
\int_{\xi}^{x} (x-t)^{n} g^{(n+1)}(t) \,{{\rm d}t}.\label{eq:remainder_def}\ee
With the change of variables $t \mapsto y = (t-\xi)/(x-\xi)$ 
in (\ref{eq:remainder_def}), we obtain
\be R_{\xi}(x) \ =  \ (x-\xi)^{n+1} T_\xi(x), \quad \text{where} \quad 
T_\xi(x) \ = \ \frac{1}{n!}  \int_{0}^{1} (1-y)^n
g^{(n+1)}(\xi+y(x-\xi)) \,{\rm d}y.\label{eq:remainder_simplified}\ee
Then, \igg{for all $x \in (\xi,b]$, $R_\xi(x) > 0 $ and  $T_\xi(x) > 0$.}  
Also, since $g\in C^{\infty}[a,b]$, we have  $T_{\xi} \in C^{\infty}[\xi,b]$ and  
$T_\xi(\xi)=\frac{1}{(n+1)!} g^{(n+1)}(\xi)>0$.

Now, recall $\alpha = 1/(n+1)$ and  define 
\begin{equation}\label{eq:def:hxi}
h_\xi(x) \ = \ \big( g(x)-g(\xi)\big)^{\alpha} \ = \ 
\big(R_\xi(x)\big)^{\alpha} \ = \ 
\big(T_\xi(x)\big)^{\alpha} (x-\xi)
.
\end{equation}
Then $h_\xi \in C^{\infty}(\xi,b]$ and,  for all 
$x \in (\xi, b]\  $, 
$h'_\xi(x) = \alpha  (g(x) - g(\xi))^{\alpha-1} g'(x) \ig{ > 0} \ $.   
Moreover since also  
$
h'_\xi(\xi)=\big(T_\xi(\xi)\big)^{\alpha}>0 $ it follows that  
$h_\xi'$ is positive valued on $[\xi,b]$ and so  
$h_\xi:[\xi,b]\to\R$ is invertible and  
$(h_\xi)^{-1} \in C^{\infty}[h_\xi(\xi), h_\xi(b)]$. Thus, inserting \eqref{eq:transform} (and $x = g^{-1}(\tau)$) into  \eqref{eq:def:hxi}, we have 
\begin{equation}
\label{eq:inv_formula}
g^{-1}(\tau)\ = \ x \ = \   h_\xi^{-1}((\tau-g(\xi))^{\alpha}).
\end{equation}

To prove the estimates \eqref{eq:est_on_inv_psi} we now   apply  
the Fa\`{a} di Bruno 
formula \eqref{eq:A} with $\phi = h_\xi^{-1}$ and 
$\psi(\tau) = (\tau - g(\xi))^\alpha$ 
to obtain derivatives of   \eqref{eq:inv_formula}. 
Consider any term in the resulting sum \eqref{eq:A}.  
Since $h_\xi^{-1}$ is smooth, the 
first factor in round brackets  is bounded, while 
the second factor  in round brackets  can be estimated by
\be
|\tau - g(\xi)|^{(\alpha - 1)m_1+ (\alpha-2)m_2 +\ldots+ (\alpha-p)m_p},
\label{eq:Cch4}
\ee
times a constant.
Recalling the remarks following \eqref{eq:A},  the index in (\ref{eq:Cch4}) is 
$\alpha \vert m \vert  - p \ \geq \ \alpha - p$, 
so \eqref{eq:est_on_inv_psi} follows. 
\end{proof}
 
\noindent 
{\em Proof of Theorem \ref{thm:derivofF}. }\ Again, without loss of generality, we work with  $\tau\in (g(\xi), g(b)]$.  
We first observe that since $g'$ is one-signed, so  is  
$(g^{-1})' $. Thus we can write  
$F = \pm  (f \circ g^{-1}) \left(g^{-1}\right)'$ 
and hence, by the Leibnitz rule, $F^{(p)}$  is a linear combination of
terms of the form 
\be 
\left(f\circ
g^{-1}\right)^{(l)} \left(g^{-1} \right)^{(p-l+1)}, \ \ \ l =
0,\ldots,p. \label{eq:D_leibnitz}\ee
Referring again to formula \eqref{eq:A},  and recalling that $f$ is smooth,   
the first term in \eqref{eq:D_leibnitz} may 
be estimated by   
a constant times $$\big\vert \prod_{j=1}^{l} ((g^{-1})^{(j)})^{m_j} \big\vert\ , $$
where $m_1 + 2m_2 + \ldots l m_l =  l$. 
Using the same argument as in the proof of Lemma \ref{lemma:psi_inv_est} this product has the estimate $\vert \tau - g(\xi)\vert ^{\alpha - l}$ 
(modulo a constant factor). 

Now, returning to the products \eqref{eq:D_leibnitz}, we see that 
for $l \neq 0$ each of these can be estimated (modulo a constant factor) by
\[ |\tau - g(\xi)|^{\alpha - l} |\tau - g(\xi)|^{\alpha - p+l-1} = 
|\tau - g(\xi)|^{2\alpha-p-1}\ . \] 
However, when $l = 0$  the bound is
$ |\tau - g(\xi)|^{\alpha-p-1}$ and since $\alpha > 0$,  the result 
\eqref{eq:F_1_deriv_est} follows. 
\hspace{0.5cm} 
$\Box$

\subsection{Accurate implementation}
\label{subsec:accuracy}
Now let us return to the computation of  \eqref{eq:capitalFdef}.  Under the assumption of 
Theorem \ref{thm:derivofF},   we write 
\begin{equation}
\label{eq:twoparts}
I_k^{[a,b]}(f,g) \ = \ \left(\int_{g(a)}^{g(\xi)} + \int_{g(\xi)}^{g(b)}\right) F(\tau) \exp({{\rm i}}k\tau) \rd\tau\ .
\end{equation}
Each of these two integrals can be transformed in an affine way to an integral 
over   $[0,1]$  so that  the singularity  is 
placed at the origin.  
The    composite FCC  algorithm given in \S \ref{section:algorithm} 
can then be applied, with  error estimates given by 
Theorem \ref{thm:beta_all}. 
For example,  consider the second integral in
\eqref{eq:twoparts}. Under the change of variable 
$\tau  = g(\xi) + c \widehat{x}$ where $c = g(b) - g(\xi)$ and $\wx \in [0,1]$
this becomes 
\begin{equation}\label{eq:troublesome}
I_{kc}^{[0,1]}(\wF) \quad \text{where} \quad   \wF(\wx) \ = \ {c}\exp({{\rm i}}kg(\xi)) F ( g(\xi) + c \wx)\ ,\end{equation} 
and, by  Theorem 
\ref{thm:derivofF},  
$\wF \in  C_\beta[0,1]$, with 
 $\beta = \alpha - 1 = -n/(n+1) \in (-1,0)$. 

In the implementation of 
the composite FCC rules for \eqref{eq:troublesome}
some care must be taken to  accurately  evaluate 
the integrand $F(g(\xi) + c \wx)$  at very small  arguments $\wx$   
(as arise in the case of  finely graded meshes).  
This is a  delicate matter since    
if $g(\xi)   \gg c \wx $,       
rounding error 
may pollute  the direct  calculation of $g(\xi) + c \wx$, in turn making $\widehat{F}(\widehat{x})$ inaccurate.  To solve 
this problem,  recall that $F$ 
 is  defined in  \eqref{eq:capitalFdef} in  terms of the   
composition of smooth functions $f$ and $g'$ with $g^{-1}$.  Our task is 
\igg{therefore} reduced to devising an  accurate evaluation of the quantity
$x\ := \ g^{-1}(g(\xi) + \epsilon)$ for small $\epsilon$.  

The required  $x$ is then a solution to the equation   $g(x) - g(\xi) = \epsilon$ and,  
recalling the proof of
Lemma  \ref{lemma:psi_inv_est}, we see that this is in turn equivalent to  
$(T_\xi(x))^\alpha (x-\xi) = \epsilon^\alpha$. Thus 
$x$ solves  the nonlinear parameter 
dependent problem 
$$
G(x,\epsilon) := (T_\xi(x))^\alpha (x - \xi) - \epsilon^\alpha {=0}\ . 
$$
Since $T_\xi(\xi)>0$ (see also the proof of Lemma \ref{lemma:psi_inv_est}), we have 
$G(\xi,0) = 0 \not = G_x(\xi,0)$ and so  
the Implicit Function Theorem implies that,  near 
$\epsilon = 0$, $x$ is a smooth function of $\epsilon^\alpha$  
and there exists a constant $C_1$ so that 
 $\vert x - \xi\vert \leq C_1 \epsilon^\alpha$, for small enough  $\epsilon$. 
Moreover $x$ is also a solution to the fixed point problem 
$$
x = \xi + \left(\frac{\epsilon}{T_\xi (x)}\right)^\alpha =: H(x)\ .
$$
Since $T_\xi(\xi)>0$ and $T_\xi$ is smooth in a neighbourhood of $\xi$,  
 it is easy to see that $H$ is Lipschitz in a ball centred on $\xi$ and 
its  Lipschitz constant is $C_2\epsilon^\alpha$ for some constant $C_2$. 
So for small enough $\epsilon$,  $x$ is the unique fixed point of $H$ and 
 fixed point iteration   converges. This suggests that for 
$\epsilon$ small,  a 
suitable approximation to 
$x$ can be chosen as 
$\tilde{x} \ :=\ H(\xi)$, with error 
$$\vert \tilde{x} -  x\vert \ = \ \vert H(\xi) - H(x) \vert \ \leq C_2 
\epsilon^{\alpha} \vert \xi - x\vert \ = \ \mathcal{O} (\epsilon^{2\alpha}) \ .$$
The approximation $\tilde{x}$ to  $g^{-1}(g(\xi) + \epsilon)$ when $\epsilon $ is small is used in the computations in \S \ref{subsec:BEM}.


\section{Numerical Experiments}
\label{sec:Numerical}

In this section, we first carry out some  numerical experiments 
which illustrate the convergence estimates of Theorem 
\ref{thm:beta_all} using computations of 
the model integral with linear oscillator:  
\be I_k^{[0,1]}(f_{\beta}) = \int_{0}^{1} f_{\beta}(x)
\exp({{\rm i}}kx)\, {\rm d}x, \ \ \ \ \ \  \label{eq:beta_int}\ee
for various $\beta\in(-1,1)$, with  $f_\beta$ defined in \eqref{eq:deffbeta} .
Then  we compute a model problem with a nonlinear oscillator motivated by the implementation of hybrid numerical-asymptotic boundary integral methods in high-frequency scattering. 

\subsection{Linear oscillator} 

\subsubsection*{\bf Experiment 1} \ \ 
Our first set of experiments studies the  case  
 $M \rightarrow \infty$, for  fixed $k$. 
From Theorem \ref{thm:beta_all} with $r = 0$,  
we see that in this case  the composite  
FCC rule for  \eqref{eq:beta_int} 
should converge with order 
$\mathcal{O}(M^{-(N+1)})$ as $M \rightarrow \infty$ provided $q > (N+1)/(\beta+1)$
for $\beta \not = 0$.  When $\beta = 0$ an additional factor of  
$\log M$ appears in the estimate. 
To illustrate this result we compute the errors $E_{k,N,M,q}(f_\beta)$  for 
$k = 1000$ with   
various $N$ and  $q = (N+1)/(\beta+1) + 0.1$ as $M$ increases.   
The exact value of \eqref{eq:beta_int} can be computed analytically and so the errors can be found exactly. 
The results for the three values  $\beta = 1/2,0,-1/4$  
are given in the three sub-tables in   
Table \ref{tab:fixedk}. The columns headed ``{\tt error}'' contain the values of
$E_{\igg{k,N,M,q}}(f_\beta)$
while the columns headed ``{\tt ratio}'' 
contain the empirical convergence rates with respect to $M$ computed 
by extrapolation 
The expected convergence rate is $N+1$ (modulo a log factor when  $\beta = 0$)
and this is given in the row marked ``{\tt expected ratio}''.  
In all cases the 
empirical convergence rate is  close to  the  
predicted rate, except when the error has almost reached machine precision  
in which case, naturally,   rather unsteady empirical convergence rates are  obtained. 
It is worth noting that in this computation some of the subintervals in 
the composite rule are very small, in fact with $N=8$ and $M = 64$ the smallest 
subinterval of the mesh is  of size about $10^{-34}$. 
Nevertheless the 
algorithm appears to show no instability and converges to machine
precision as $M$ increases. 
\begin{table}[hbt]
\centering
{\footnotesize \tt \addtolength{\tabcolsep}{-2pt}
\begin{tabular}{|c||c|c|c|c|c|c|} \hline 
$\beta = 1/2$ &\multicolumn{2}{|c|}{$N = 4$}&
\multicolumn{2}{|c|}{$N = 6$}& \multicolumn{2}{|c|}{$N = 8$}\\
expected ratio&\multicolumn{2}{|c|}{$5$}&
\multicolumn{2}{|c|}{$7$}&\multicolumn{2}{|c|}{$9$}\\
\hline
$M$ & error & ratio  & error & ratio  & 
 error & ratio   \\ 
\hline
  8 & 4.3e-006 &   & 5.2e-008 &    & 1.7e-009 &    \\ 
 16 & 9.5e-008 &   5.49 & 5.7e-010 &   6.50 & 6.6e-012 &   8.06 \\ 
 32 & 2.9e-009 &   5.03 & 2.0e-012 &   8.13 & 1.0e-014 &   9.30 \\ 
 64 & 8.1e-011 &   5.17 & 2.3e-014 &   6.50 & 1.3e-016 &   6.37 \\ 
  \hline
\end{tabular}}

\vspace{0.1cm}

{\footnotesize \tt \addtolength{\tabcolsep}{-2pt}
\begin{tabular}{|c||c|c|c|c|c|c|} \hline 
$\beta = 0$ &\multicolumn{2}{|c|}{$N = 4$}&
\multicolumn{2}{|c|}{$N = 6$}& \multicolumn{2}{|c|}{$N = 8$}\\
expected ratio&\multicolumn{2}{|c|}{$5$}&
\multicolumn{2}{|c|}{$7$}&\multicolumn{2}{|c|}{$9$}\\
\hline
$M$ & error & ratio  & error & ratio  & 
 error & ratio \\ 
\hline
  8 & 2.7e-004 &    & 7.9e-006 &   & 1.0e-006 &    \\ 
 16 & 1.0e-005 &   4.70 & 7.3e-008 &   6.77 & 2.2e-009 &   8.82 \\ 
 32 & 4.0e-007 &   4.67 & 7.4e-010 &   6.62 & 3.0e-012 &   9.53 \\ 
 64 & 1.4e-008 &   4.87 & 3.8e-012 &   7.59 & 1.9e-015 &  10.59 \\ 
     \hline
\end{tabular}}

\vspace{0.1cm}

{\footnotesize \tt \addtolength{\tabcolsep}{-2pt}
\begin{tabular}{|c||c|c|c|c|c|c|} \hline 
$\beta = -1/4$ &\multicolumn{2}{|c|}{$N = 4$}&
\multicolumn{2}{|c|}{$N = 6$}& \multicolumn{2}{|c|}{$N = 8$} \\expected ratio&\multicolumn{2}{|c|}{$5$}&
\multicolumn{2}{|c|}{$7$}&\multicolumn{2}{|c|}{$9$}\\
\hline
$M$ & error & ratio  & error & ratio  & 
 error & ratio   \\ 
\hline
  8 & 4.5e-005 &    & 1.6e-005 &    & 6.0e-006 &    \\ 
 16 & 2.6e-006 &   4.10 & 8.0e-008 &   7.62 & 2.0e-008 &   8.22 \\ 
 32 & 1.9e-008 &   7.15 & 9.3e-010 &   6.43 & 1.1e-011 &  10.91 \\ 
 64 & 1.9e-009 &   3.33 & 3.9e-012 &   7.88 & 2.9e-014 &   8.51 \\ 
\hline
\end{tabular}}
\caption{ \label{tab:fixedk} Numerical Results for Experiment 1}
\end{table}

\begin{table*}[htbp]
\centering
{\footnotesize \tt \addtolength{\tabcolsep}{-2pt}
\begin{tabular}{|c||c|c|c|c|c|c|c|c|} \hline   
&\multicolumn{2}{c|}{$\beta = 1/8$}&\multicolumn{2}{c|}{$\beta = 1/4$}& \multicolumn{2}{c|}{$\beta = 1/2$ } &
\multicolumn{2}{c|}{$\beta = 3/4$ }\\ best expected ratio & \multicolumn{2}{|c|}{
0.86}&\multicolumn{2}{|c|}{ 1.00}& \multicolumn{2}{|c|}{ 1.27} & \multicolumn{2}{|c|}{1.55 }\\   \hline    
$k_i$ & error & ratio &  error & ratio &  error & ratio &
 error & ratio  \\
\hline
 $10^3$ & 4.9e-006 &   & 4.0e-006 &    & 1.2e-006 &    & 2.2e-007 &    \\ 
 $10^4$ & 4.7e-007 &   1.02 & 2.7e-007 &   1.16 & 4.5e-008 &   1.43 & 4.5e-009 &   1.68 \\ 
 $10^5$ & 5.7e-008 &   0.91 & 2.6e-008 &   1.03 & 2.3e-009 &   1.29 & 1.1e-010 &   1.62 \\ 
 $10^6$ & 1.2e-008 &   0.68 & 3.8e-009 &   0.83 & 1.8e-010 &   1.11 & 4.9e-012 &   1.35 \\ 
 $10^7$ & 1.3e-009 &   0.95 & 2.5e-010 &   1.18 & 4.4e-012 &   1.60 & 7.1e-014 &   1.84 \\ 
 \hline
\end{tabular}}

\vspace{0.1cm}

{\footnotesize \tt \addtolength{\tabcolsep}{-2pt}
\begin{tabular}{|c||c|c|c|c|c|c|c|c|} \hline 
&\multicolumn{2}{|c|}{$\beta = -1/16$}&
\multicolumn{2}{|c|}{$\beta =- 1/8$}& \multicolumn{2}{|c|}{$\beta = -1/4$}
&\multicolumn{2}{|c|}{$\beta = -1/2$} \\   
best expected ratio &\multicolumn{2}{|c|}{$ 0.96$}& 
\multicolumn{2}{|c|}{$0.59$}&\multicolumn{2}{|c|}{$0.45$}
&\multicolumn{2}{|c|}{$0.18$} \\ 

\hline
$k$ & error & ratio & error & ratio  &  error & ratio  & error & ratio  \\ 
\hline
  $10^3$ & 9.3e-006 &    & 2.9e-005 &    & 1.4e-004 &    & 1.6e-003 &    \\ 
  $10^4$ & 1.5e-006 &   0.81 & 5.4e-006 &   0.73 & 3.7e-005 &   0.57 & 1.2e-003 &   0.13 \\ 
  $10^5$ & 2.5e-007 &   0.76 & 1.0e-006 &   0.71 & 8.6e-006 &   0.64 & 4.8e-004 &   0.39 \\ 
  $10^6$ & 9.0e-008 &   0.45 & 4.4e-007 &   0.37 & 5.1e-006 &   0.23 & 3.4e-004 &   0.14 \\ 
  $10^7$ & 2.3e-008 &   0.60 & 1.5e-007 &   0.48 & 3.1e-006 &   0.22 & 8.0e-004 &  -0.37 \\ 
      \hline
\end{tabular}}
\caption{\label{tab:kvarying} Numerical Results for Experiment 2}
\end{table*}

\subsubsection*{Experiment 2} \ \
Here we fix $M=10$, $N=3$ and $q=12$ and we study convergence as  
$k$ increases,  for various  $\beta$. 
In Table \ref{tab:kvarying},   
the columns headed ``{\tt ratio}'' contain the empirical convergence 
rates with respect to $k$ computed by extrapolation.  
From Theorem \ref{thm:beta_all},
 we see that the composite FCC rule for \eqref{eq:beta_int} should converge with order $\mathcal{O}(k^{-r})$ as $k \rightarrow \infty$ where $r < (q(\beta+1) - N-1)/(q-1)$. In the row marked ``{\tt best expected ratio}'' 
this upper bound on $r$ is given for each $\beta$, using our choice of $N, q$. 
We see from Table \ref{tab:kvarying} that the empirical convergence rate as 
$k$ increases for   $\beta>0$ is close to  the  theoretically 
predicted best rate of convergence. When  $\beta<0$, 
the empirical rates are a bit slower that the theoretical best rate. 

\subsubsection*{Experiment  3} \ \ 
In Table
\ref{tab:kvarlog}    we study the     computation   of
\eqref{eq:beta_int} with $f_{\beta}(x) = \log x$, for $M=12$ and $N=3$
as $k$  increases for various $q$.  
 For each  value of  $q$  and  $N$, the  error of the  composite FCC
rule should converge with order  $\mathcal{O}(k^{-r})$, 
with $r < (q - N-1)/(q-1)$. The row  marked ``{\tt best  expected ratio}''
contains the upper bound on $r$ while the   
columns marked  ``{\tt ratio}'' contain the empirical 
convergence rates .   The table shows that when when $q =
12$ the  empirical    convergence rate is close the  the theoretical 
best rate.  When $q = 4$, the  empirical convergence rate is
even better than  the best expected rate which  in this case indicates
that no convergence should be observed at all. On the other hand, when
$q  =  8$  and  $q   =  16$  rather  unsteady  convergence  rates  are
obtained. However, when  $q = 8$ as $k$  increases the empirical rates
become  bounded by  the best  expected rate,  while for  $q =  16$ the
empirical rates are either bounded  by or are slightly better than the
the best expected rate.  
\begin{table}[htbp]
\centering
{\footnotesize \tt \addtolength{\tabcolsep}{-2pt}
\begin{tabular}{|c||c|c|c|c|c|c|c|c|} \hline 
 &\multicolumn{2}{|c|}{}& \multicolumn{2}{|c|}{} &
\multicolumn{2}{|c|}{}& \multicolumn{2}{|c|}{} \\ 
& \multicolumn{2}{|c|}{$q = 4$}
&\multicolumn{2}{|c|}{$q = 8$}&
\multicolumn{2}{|c|}{$q = 12$}& \multicolumn{2}{|c|}{$q = 16$}
\\   
best expected ratio &\multicolumn{2}{|c|}{$0$} &\multicolumn{2}{|c|}{$0.57$}&
\multicolumn{2}{|c|}{$0.73$}&\multicolumn{2}{|c|}{$0.80$}
\\ 
\hline
$k$ & error & ratio & error & ratio & 
 error & ratio & error & ratio \\ 
\hline
 && && && && \\    
 $10^1$  & 5.5e-004 &   0.00 & 1.5e-004 &   0.00 & 1.1e-003 &   0.00 & 3.6e-003 &   0.00 \\ 
 $10^2$  & 5.2e-004 &   0.02 & 5.6e-005 &   0.42 & 2.2e-004 &   0.69 & 3.5e-004 &   1.01 \\ 
 $10^3$  & 5.2e-004 &   0.00 & 3.3e-005 &   0.24 & 3.8e-005 &   0.75 & 1.0e-004 &   0.53 \\ 
 $10^4$  & 5.0e-004 &   0.02 & 6.7e-006 &   0.69 & 7.0e-006 &   0.74 & 8.4e-006 &   1.09 \\ 
 $10^5$  & 1.4e-004 &   0.55 & 9.1e-007 &   0.87 & 1.1e-006 &   0.79 & 1.9e-006 &   0.64 \\ 
 $10^6$  & 2.0e-005 &   0.84 & 3.9e-007 &   0.37 & 2.0e-007 &   0.74 & 2.4e-007 &   0.91 \\ 
 $10^7$  & 1.9e-006 &   1.04 & 1.3e-007 &   0.47 & 5.1e-008 &   0.60 & 8.5e-008 &   0.45 \\
 \hline
\end{tabular}}
\caption{  \label{tab:kvarlog} Numerical results for Experiment 3}
\end{table} 

\subsubsection*{Experiment 4} \ \ 
Here we  illustrate the power of the composite FCC rule compared to   
the non-composite version for computing \eqref{eq:beta_int} when  
$\beta = 1/2$. 
The parameter $N$ for the non-composite FCC rule takes values $N_i = \{24 \times 2^{i}, i = 0,1,2,3\}$, while for the composite rule, we fix parameters $q = 12$ and $M = 6$ and take $N_i = \{4 \times 2^{i}, i = 0,1,2,3\}$. The total number of function evaluations in both cases is therefore the same. 
The  superiority of the composite version is clearly seen. 

\begin{table}[htbp]
\centering
{\footnotesize \tt \addtolength{\tabcolsep}{-2pt}
\begin{tabular}{|c||c|c||c||c|c|} \hline  
&\multicolumn{2}{|c||}{non-composite FCC }&& \multicolumn{2}{|c|}{composite FCC }
\\
&\multicolumn{2}{|c||}{$M= 1$, $q = 1$}&& \multicolumn{2}{|c|}{$M= 6$, $q = 12$} \\
\hline
$N$&$k = 400$&$k = 1600$& $M \times N$& $k = 400$  &$k = 1600$ \\   \hline    
 24 & 9.2e-004  & 4.5e-005  & 6$\times$4  & 1.5e-005  & 1.0e-006 \\
48 & 5.9e-004  & 4.4e-005  & 6$\times$8  & 8.4e-007  & 2.3e-007 \\
96 & 1.8e-004  & 4.2e-005  & 6$\times$16 & 1.5e-008  & 1.5e-008 \\
192& 9.7e-005  & 2.6e-005  & 6$\times$32 & 5.5e-012  & 3.3e-009 \\
\hline
\end{tabular}} 
\caption{\label{tab:experiment4} Numerical results for Experiment 4}
\end{table}

\subsection{An example from boundary integral methods in high-frequency scattering}
\label{subsec:BEM} 
{


Finally,  we describe an application to the 
computation of acoustic scattering at high frequency by numerical-asymptotic methods. 
When an   incident  plane  wave   $\exp({\rm i}k
\bm{x}.\widehat{{\bf d}})$ is scattered by a smooth convex  sound-soft obstacle 
with boundary $\Gamma$, the scattered field can be computed by 
solving the integral equation
\[
 \int_{\Gamma}\frac{\rm i}{4} H_0^{(1)}(k|{\bf x}-{\bf y}|)v({\bf y})\,{\rm
d}s({\bf y})=\exp({\rm i}k
\bf{x}.\widehat{{\bf d}})\ , \ {\bf x} \in \Gamma. 
\]
where $H_0^{(1)}$ is the Hankel function of the first kind of order $0$. 
(In fact this problem is not well-posed for all values of $k$ and in practice a 
related ``combined potential'' formulation is used. However this formulation  
illustrates the essential quadrature challenge which arises 
in all formulations. ) 
Even for moderate values of $k$ is useful to apply the ``physical optics approximation'' which amounts to writing   
$
v({\bf y})=V({\bf y})\exp({{\rm i}}k{\bf
  y}.\widehat{\bf d}). 
$
and computing the less oscillatory component  $V$ rather than the highly oscillatory  $v$ - see \cite{DoGrSm:07}, \cite{ChGrLaSp:12}.

Using the fact that 
$H_0^{1} (kr)\exp(-{{\rm i}}kr)$ 
is a non-oscillatory function, 
smooth for $r>0$ but with a logarithmic singularity at $r=0$ 
and introducing   a smooth parameterization ${\bf
x}:[0,2\pi]\to\Gamma$,  the
problem above can be reformulated, \igg{for $s \in [0, 2\pi]$} as
\begin{equation}\label{eq:01:BEM}
 \int_{0}^{2\pi}  M_k(s,t)\exp({{\rm i}}k\Psi_{[s]}(t)) V(t)\,{\rm
d}t=1, \quad \Psi_{[s]}(t):=|{\bf x}(s)- {\bf x}(t)|-({\bf x}(s)-{\bf
x}(t)) . \widehat{\bf d}\ .
\end{equation}
The function $M_k(s,t)$ is non-oscillating and smooth 
except as $t=s$ where a logarithmic singularity occurs. (More details are in \cite{KimPhD:2011}.)  
It can be proved that 
when $s$ is chosen so that ${\bf x}(s)$ is in the ``illuminated'' part of 
$\Gamma$ (where the incident waves hits the obstacle),  there is only 
one 
stationary point,
i.e., there exists a unique 
$t$ so that $\big(\Psi_{[s]}\big)'(t)=0$. Moreover, ${\bf x}(t)$ is a
point in the shadow part. Conversely, if ${\bf x}(s)$ lies in the shadow, we
have three stationary
points, with two of them in the shadow and one in the illuminated part. 
Thus \eqref{eq:01:BEM} is a good example for application of  the methods of \S 
\ref{sec:Composite_Stationary}. 

As an illustration we compute  the integral  in \eqref{eq:01:BEM} 
when $\Gamma$ is the unit circle, $\mathbf{x}(s) = (\cos s , \sin s)$ 
and set $\widehat{\bf d}=(1,0)$ so that 
\[
\Psi_{[s]}(t)=2\Big|\sin\left(\frac{s-t}{2}\right)\Big|-\cos s+\cos t. 
\]
The integral is computed using the following  strategy.  First,  
$[0,2\pi]$ is divided into subintervals so that each of them contains 
at most one  point which is either the singular point $s$ or a
stationary point. Next, any subinterval of length greater than one is split
into \igg{two subintervals of equal length} and  this process is continued  
until we obtain a subdivision of $[0,2\pi]$ in say $J$ subintervals of length smaller than 1. This is done to avoid working with
graded meshes on relatively long intervals. Then, with the change of
variable $\tau=\Psi_{[s]}(t)$, we have to approximate the integral
\[
\int_{\Psi_{[s]}(a_j)}^{\Psi_{[s]}(b_j)}  F_{s}(\tau)\exp({\rm
i}k\tau)\:{\rm d}\tau, \quad 
 F_{s}(\tau):=M_k(s,\Psi_{[s]}^{-1}(\tau) )
V(\Psi_{[s]}^{-1}(\tau))\frac{1}{(\Psi_\igg{[s]})'(\Psi_\igg{[s]}^{-1}(\tau))}, \quad j=1,\ldots,J.
\]
Given $L$ a positive integer, we compute the approximation of these integrals as
follows: If $[a_j,b_j]$ does not contain $s$ or a stationary point, the  function
\igg{$F_s$} is smooth, and therefore we can apply the simple 
{Filon} Clenshaw-Curtis
rule with $\min\{L+1,129\}$ points. If 
$s\in[a_j,b_j]$,  after an affine change of variables, 
$F_s$ belongs to ${\cal C}_0^ \infty[0,1]$. On the other hand, if either  $a_j$ or $b_j$ is a stationary point of order $n\ge 1$,  \ig{$ {F\in} {\cal C}_{-n/(n+1)}^ \infty[0,1]$} cf. Theorem \ref{thm:derivofF}. In both cases, we introduce  meshes appropriately graded towards the singularity, according Theorem \ref{thm:beta_all} with
$L$ subintervals and apply the composite FCC rule with $N+1$
points. 
In our experiment we have taken  $V\equiv 1$ and $s=3\pi/4$, which corresponds
to a point in the illuminated part. The only stationary point   $t=23\pi/24$
is of order $1$. 

In Table \ref{table:BEM:01} we display the results obtained with $N=6$, 
\igg{$q=(N+1)/(1+\beta)+1/10$}, which corresponds to $r=0$ in Theorem \ref{thm:beta_all}. Here  $\beta=0$ or $-1/2$ depending on the singularity of the integrand. 
From Theorem \ref{thm:beta_all} we can expect that the error decreases with  
as ${\cal O}(L^{-7})$ {(but there is no predicted decay with
  respect to  $k$ in this case).}
In Table \ref{table:BEM:02}, we give results for  $N=4$ and 
$q=(N+1)/(1+\beta-1/4)+1/10$, 
leading  to the error estimate   ${\cal
O}(k^{-1/4}L^{-4.75})$. Although the convergence in Table \ref{table:BEM:02}  
is irregular,  in contrast to 
Table \ref{table:BEM:01}, the error decreases with  $k$ as predicted by the theory.

In this experiment over   90\% of the CPU time  is spent in  computing the
change of variable, since any evaluation in $\tau$ requires the solution of  a
non-linear equation.  This is carried out in  our implementation using
the {\tt fzero}  of Matlab. 
In
the case  that \eqref{eq:01:BEM}  has to  be computed  for  many different
functions $V$,  for instance  in assembling the  matrix of  a Boundary
Element  Method,  these  calculations  have  to  be  done  only  once,
resulting in a  speed up of the  the method. We run our  programs in a
modest laptop, a Core 2 Duo with  4Gb of Ram memory, and show in Table
\ref{table:BEM:03} the CPU time  required to construct some columns of
Table \ref{table:BEM:01}.  We clearly see  that the cost is  linear in
$L$.

 }
\begin{table}[htbp]
\centering
{\footnotesize \tt \addtolength{\tabcolsep}{-2pt}
\begin{tabular}{|c||c|c|c|c|c|c|c|c|c|c|} 
\hline 
 &\multicolumn{2}{|c|}{}& \multicolumn{2}{|c|}{} &
\multicolumn{2}{|c|}{} & \multicolumn{2}{|c|}{} &
\multicolumn{2}{|c|}{} \\
&\multicolumn{2}{|c|}{$k = 10$}&
\multicolumn{2}{|c|}{$k = 100$}& \multicolumn{2}{|c|}{$k = 1,000$}
&\multicolumn{2}{|c|}{$k = 10,000$} &\multicolumn{2}{|c|}{$k = 100,000$}\\   
 &\multicolumn{2}{|c|}{}& \multicolumn{2}{|c|}{} &
\multicolumn{2}{|c|}{} & \multicolumn{2}{|c|}{} &
\multicolumn{2}{|c|}{}  \\ 
\hline &&&&&&&&& &\\
$L$ & error & ratio & error & ratio & 
 error & ratio & error & ratio & error & ratio  \\ 
\hline
 && && && && &&\\
12 &  4.5e-07  &       &  2.4e-07  &       & 1.2e-07  &     &  1.6e-08 &     
  &7.9e-08 &\\
24 &  5.0e-09  &  6.5  &  1.0e-09  &  7.9  & 1.9e-09  & 6.0 &  1.3e-09 & 3.6 &
 1.0e-09 & 6.2 \\
48 &  4.6e-11  &  6.8  &  5.1e-12  &  7.6  & 1.4e-11  & 7.1 &  8.1e-12 & 7.4 &
9.0e-12 &   6.9 \\
96 &  2.3e-13  &  7.6  &  1.1e-13  &  5.5  & 1.3e-13  & 6.8 &  1.0e-13 & 6.3 &
4.4e-13  &   4.4\\
192&  8.3e-15  &  4.8  &  6.4e-15  &  4.1  & 5.1e-15  & 4.6 &  1.1e-15 & 6.6&
 1.1e-15  &    8.6\\
\hline
\end{tabular}}
\caption{\label{table:BEM:01}\ig{Results for 
$N=6, M=L$ and  $q=(N+1)$ for  intervals with log singularities and
$q=(N+1)/(1-1/2)$ for integrals having a stationary point (so $r=0$)}. }
\end{table}

\begin{table}[htbp]
\centering
{\footnotesize \tt \addtolength{\tabcolsep}{-2pt}
\begin{tabular}{|c||c|c|c|c|c|c|c|c|c|c|} 
\hline 
 &\multicolumn{2}{|c|}{}& \multicolumn{2}{|c|}{} &
\multicolumn{2}{|c|}{} & \multicolumn{2}{|c|}{} &
\multicolumn{2}{|c|}{} \\
&\multicolumn{2}{|c|}{$k = 10$}&
\multicolumn{2}{|c|}{$k = 100$}& \multicolumn{2}{|c|}{$k = 1,000$}
&\multicolumn{2}{|c|}{$k = 10,000$} &\multicolumn{2}{|c|}{$k = 100,000$}\\   
 &\multicolumn{2}{|c|}{}& \multicolumn{2}{|c|}{} &
\multicolumn{2}{|c|}{} & \multicolumn{2}{|c|}{} &
\multicolumn{2}{|c|}{}  \\ 
\hline &&&&&&&&& &\\
$L$ & error & ratio & error & ratio & 
 error & ratio & error & ratio & error & ratio  \\ 
\hline
 && && && && &&\\  
12&4.7e-05 &     & 1.9e-05 &     & 2.4e-06 &     & 8.2e-07 &    & 1.8e-07 &\\
24&6.0e-07 &6.3 & 1.7e-07 & 6.8 & 1.2e-07 &  4.3& 1.4e-08 & 5.9& 1.0e-08 &4.1\\
48&1.2e-08 &5.7 & 1.4e-08 & 3.6 & 3.8e-09 &  5.0& 2.4e-10 & 5.9& 2.0e-10 &5.7\\
96&3.1e-10 &5.2 & 3.9e-10 & 5.1 & 5.3e-11 &  6.2& 7.0e-12 & 5.1& 1.9e-12 &6.7\\
192&1.4e-11&4.5 & 3.3e-11 & 3.6 & 2.8e-11 &  0.9& 3.7e-12 & 0.9& 3.0e-13 &2.7\\
\hline
\end{tabular}}
\caption{\label{table:BEM:02} \ig{Results for  $N=4$,
$M=L$ and $q=(N+1)/(1-1/4)$ for intervals with log singularities and
$q=(N+1)/(1-1/2-1/4) $ for intervals having a stationary point (so $r=1/4$).} 
}
\end{table}

\begin{table}[htbp]
\centering
{\footnotesize \tt \addtolength{\tabcolsep}{-2pt}
\begin{tabular}{|c|c|c|c|} 
\hline &&& \\
L &\parbox{2cm}{\centerline{$k=10$}} &\parbox{2cm}{\centerline{$k=1000$}}  & \parbox{2cm}{\centerline{$k=10,000$}}   \\ 
\hline & & & \\
12     & 0.77''&0.83'' &0.78''\\
24     &1.41'' &1.34'' &1.45''\\
48     &2.44'' &2.49'' &2.42''\\
96     &4.50'' &4.64'' &4.48''\\
192    &8.48'' &8.46'' &8.37''\\
\hline  
\end{tabular}}
\caption{\label{table:BEM:03} CPU time consumed for computing Table \ref{table:BEM:01} }
\end{table}

\paragraph{Acknowledgement} The first author is partially supported by
Project MTM2010-21037 (Spain)  and EPSRC Grant
EP/F06795X/1 (UK). The second author was supported by EPSRC Grant
EP/F06795X/1. The third author was supported by a Postgraduate
Studentship from the University of Bath and EPSRC Grant
EP/F06795X/1.


\pagebreak

\bibliographystyle{plain}

\end{document}